\documentclass{amsart}
\usepackage{amsmath, amsthm, amsfonts, amssymb}
\usepackage{mathrsfs}
\usepackage{microtype}
\usepackage[utf8]{inputenc}
\providecommand{\noopsort[1]{}}
\usepackage{bbm}
\usepackage{dsfont}
\usepackage{eucal}
\usepackage{ifthen}
\usepackage{nicefrac}
\numberwithin{equation}{section}
\usepackage{a4}
\usepackage[active]{srcltx}
\usepackage{verbatim}

\usepackage[shortlabels]{enumitem}
\setlist{leftmargin=*}
\setlist[1]{labelindent=1.2\parindent}

\usepackage[colorlinks,pdftex]{hyperref}
\hypersetup{
linkcolor=black,
citecolor=black,
pdftitle={},
pdfauthor={},
pdfkeywords={},
}

\renewcommand\labelenumi{\textup{(\roman{enumi})}}
\renewcommand\theenumi\labelenumi
\renewcommand\labelenumii{(\alph{enumii})}
\renewcommand\theenumii\labelenumii

\renewcommand\theenumii\labelenumii

\newtheorem{thm}{Theorem}[section]
\newtheorem{cor}[thm]{Corollary}
\newtheorem{prop}[thm]{Proposition}
\newtheorem{lem}[thm]{Lemma}

\theoremstyle{remark}

\newtheorem{rem}[thm]{Remark}

\newtheorem{hyp}[thm]{Hypothesis}
\newtheorem{example}[thm]{Example}

\theoremstyle{definition}
\newtheorem{defn}[thm]{Definition}
\newcommand{\coloneqq}{\mathrel{\mathop:}=}

\renewcommand{\Re}{{\rm Re}\,}

\newcommand{\eps}{\varepsilon}

\newcommand{\one}{\mathds{1}}

\newcommand{\R}{\mathds{R}}
\newcommand{\C}{\mathds{C}}

\newcommand{\N}{\mathds{N}}

\newcommand{\I}{\mathds{1}}

\newcommand{\cB}{\mathcal{B}}
\newcommand{\cM}{\mathcal{M}}
\newcommand{\cD}{\mathcal{D}}
\newcommand{\cL}{\mathcal{L}}
\newcommand{\cA}{\mathcal{A}}

\newcommand{\Bb}{\mathcal{B}_b}
\newcommand{\Cb}{\mathcal{C}_b}
\newcommand{\Co}{\mathcal{C}_\infty}
\newcommand{\Bo}{\mathcal{B}_\infty}
\newcommand{\Mp}{\mathcal{M}^+}
\newcommand{\Cc}{\mathcal{C}_c}

\newcommand{\PP}{\mathds{P}}
\newcommand{\E}{\mathds{E}}

\newcommand{\la}{\langle}
\newcommand{\ra}{\rangle}

\DeclareMathOperator{\id}{\mathrm{id}}
\DeclareMathOperator{\tr}{\mathrm{tr}}
\newcommand{\weak}{\rightharpoonup}
\newcommand{\bp}{\stackrel{\mathrm{bp}}{\to}}

\makeatletter
\@namedef{subjclassname@2020}{%
  \textup{2020} Mathematics Subject Classification}
\makeatother

\begin{document}
\title[Feller generators with measurable lower order terms]{Feller generators\\ with measurable lower order terms}

\author[F.~K\"{u}hn]{Franziska K\"{u}hn}
\address[F.~K\"{u}hn]{TU Dresden, Fakult\"at Mathematik, Institut f\"{u}r Mathematische Stochastik, 01062 Dresden, Germany. E-Mail: \textnormal{franziska.kuehn1@tu-dresden.de}}

\author[M. Kunze]{Markus Kunze}
\address[M. Kunze]{Universit\"at Konstanz, Fachbereich Mathematik und Statistik, Fach 193, 78357 Konstanz, Germany. E-Mail: \textnormal{markus.kunze@uni-konstanz.de}}

\subjclass[2020]{60G53; 47A55; 47G20; 60J35; 60H10}
\keywords{Feller semigroup; perturbation; L\'evy-type operator; measurable coefficients; martingale problem; stochastic differential equation}

\begin{abstract}
We study perturbations of Feller generators under `lower order terms' with measurable coefficients. We investigate which properties of the original semigroup -- such as positivity, conservativeness and the Feller property -- are passed to the perturbed semigroup. We give several examples and discuss applications in the theory of martingale problems and stochastic differential equations with measurable coefficients.
\end{abstract}

\maketitle

\section{Introduction} 

Given two operators $A$ and $B$, it is a classical question to ask how the sum $A+B$ is defined and which properties of $A$ are preserved under the perturbation $B$, see e.g.\ \cite{kato95,ban06}. In this article, we consider this problem in the context of Feller semigroups and generators. We are interested in constructing Feller generators with measurable lower order terms, and this means that the perturbation does not take values in the space of continuous functions. In consequence, classical perturbation results from semigroup theory do not apply. \smallskip

Let $(T(t))_{t \geq 0}$ be a Feller semigroup, that is, a sub Markovian, strongly continuous semigroup $T(t): \Co(\R^d) \to \Co(\R^d)$ on the space $\Co(\R^d)$ of continuous functions vanishing at infinity. If the domain of the (infinitesimal) generator $A$ contains the test functions $\Cc^{\infty}(\R^d)$, then the Courr\`ege--van Waldenfels theorem, see e.g.\ \cite[Thm.\ 2.21]{ltp3}, shows that $\cA := A|_{\Cc^{\infty}(\R^d)}$ is a \emph{L\'evy-type operator}, i.e.\ an integro-differential operator of the form 
\begin{align}\label{eq.generalop}\begin{aligned}
	\cA f(x) &= b(x) \nabla f(x) + \frac{1}{2} \tr(Q(x) \nabla^2 f(x)) \\
	&\quad + \int_{\R^d \setminus \{0\}} (f(x+y)-f(x)-\nabla f(x) \cdot \I_{(0,1)}(|y|)) \, \nu(x,dy).
\end{aligned} \end{align}
Here $(b(x), Q(x), \nu (x, \cdot))_{x\in \R^d}$ are the \emph{(infinitesimal) characteristics} consisting of the \emph{drift coefficients} $b= (b_j): \R^d\to \R^d$, the \emph{diffusion coefficients} $Q=(q_{ij}) : \R^d\to \R^{d\times d}$ and the \emph{jumping kernels} $\nu : \R^d \to \cM^+(\R^d)$. We are interested in the following questions: If a L\'evy-type operator $B$ is a lower order perturbation of $A$, then under which conditions is (a realization of) $A+B$ the generator of a semigroup and which properties does the perturbed semigroup inherit from the original semigroup $(T(t))_{t \geq 0}$? Is the martingale problem for $A+B$ well-posed? Classical perturbations results can be used to tackle these questions if $B$ maps $D(A)$ into $\Co(\R^d)$, which in particular implies  that the infinitesimal characteristics of $B$ need to depend continuously on $x$. In this article, we investigate perturbations $B$ whose characteristics depend merely measurably on $x$. Allowing for discontinuous characteristics of $B$ creates a number of issues and subtleties. To give an example: If $B$ maps $D(A)$ into $\Co(\R^d)$, then the sub Markovianity of the perturbed semigroup can be verified using the positive maximum principle; this does not work any longer if $B$ has discontinuous coefficients, i.e.\ if we  work on the space $\Bb(\R^d)$ of bounded measurable functions rather than $\Co(\R^d)$. In fact, establishing the sub Markov property of the perturbed semigroup turns out to be rather delicate point. At the same time, being sub Markovian is crucial for the applications in the theory of stochastic processes, which we are interested in. To establish sub Markovianity, we will use an approximation argument which is of independent interest, e.g.\ it can be used to prove continuous dependence of solutions of certain martingale problems on the coefficients (see Theorem \ref{t.sg.conv}). Let us mention that our questions can be formulated equivalently in the context of pseudo-differential operators. Namely, the L\'evy-type operator $\cA$ can be seen alternatively as a pseudo-differential operator
\begin{equation}
	\cA f(x) = - \int_{\R^d} q(x,\xi) e^{ix \cdot \xi} \hat{f}(\xi) \, d\xi, \qquad f \in \Cc^{\infty}(\R^d),\;x \in \R^d, \label{eq.pseudo}
\end{equation}
where $\hat{f}$ is the Fourier transform and
\begin{equation}
	q(x,\xi) := -ib(x) \cdot \xi + \frac{1}{2} \xi \cdot Q(x) \xi + \int_{\R^d \setminus \{0\}} (1-e^{iy \cdot \xi}+iy \cdot \xi \I_{(0,1)}(|y|)) \, \nu(x,dy)
	\label{eq.symbol}
\end{equation}
is the so-called \emph{symbol} of $\cA$. We study under which conditions the pseudo-differential operator with symbol $q(x,\xi)+p(x,\xi)$ gives rise to a Feller semigroup if the symbol $p(x,\xi)$ of the perturbation $B$ depends merely measurably on $x$. \smallskip

If the symbol -- or equivalently, the characteristics -- of a L\'evy-type operator satisfies suitable smoothness conditions, then general results from symbolic calculus show that the closure of $(\cA,\Cc^{\infty}(\R^d))$ is the generator of a Feller semigroup, see e.g.\ \cite{hoh,J05}. Under the milder assumption that the coefficients are H\"older continuous, the situation is already much more complicated and there is an immense literature on the question whether the L\'evy-type operator $\cA$ gives rise to a Feller semigroup, see e.g. \cite{ltp3,hoh,J02,J05,lm6} for a survey. None of these results applies in our framework since we are dealing with discontinuous coefficients. For the particular case that the diffusion coefficient $Q$ is strictly elliptic, there are general results in the literature which allow discontinuous coefficients, see e.g.\ \cite[Theorem 2.1.43]{J02} and also \cite{taira1,taira2} for processes on bounded domains. Moreover, the well-posedness of the $(\cA,\Cc^{\infty}(\R^d))$-martingale problem is well studied in this case, cf.\ \cite{stroock75}. No such general results are available if the operator has a vanishing diffusion component $Q \equiv 0$. There are some perturbation results for martingale problems which allow discontinuous coefficients, including bounded perturbations \cite[Section 4.10]{ek} and perturbations of L\'evy generators e.g.\ in \cite{chen15,komatsu84,mik14,peng}. Note that already the existence of a solution to the martingale-problem can be highly non-trivial if the coefficients are discontinuous: While there are general existence results for L\'evy-type operators with continuous coefficients, cf.\ \cite[Theorem 3.15]{hoh}, there are no such results for the discontinuous framework; see \cite{kuehn-mp} and the references therein. Let us point out that the existence of (unique) solutions to martingale problems associated with L\'evy-type operators can be used to deduce the existence of (unique) weak solutions to L\'evy-driven stochastic differential equations (SDEs)
\begin{equation*}
	dX_t = b(X_{t-}) \, dt + \sigma(X_{t-}) \, dL_t,
\end{equation*}
see Subsection~\ref{s.sde} for details. In particular, drift(-type) perturbations of L\'evy processes (i.e.\ $\sigma \equiv 1$) have been studied quite intensively, \cite{portenko94,tanaka,chen15,kim14,kin20} to mention just a few classical and recent works.

The following general result on Feller semigroups with measurable lower order terms is obtained by combining our main results Theorem \ref{t.perturbeuclidean} and Theorem \ref{t.mp.wellposed}.

\begin{thm} \label{t.generic} 
	Let $(T(t))_{t \geq 0}$ be a Feller semigroup with generator $A$ such that $\Cc^\infty(\R^d)$ is a core for $A$. Assume that there are $\rho \in (0,2)$ and $\varphi \in L^1(0,1)$ such that 
	\begin{equation*}
		\|T(t) f\|_{\Cb^{\rho}} \leq \varphi(t)\|f\|_{\infty}, \qquad t \in (0,1),\, f \in \Bb(\R^d),
	\end{equation*}
	Let $\hat{B}$ be a L\'evy-type operator
	\begin{equation}
		\hat{B}f(x) = b(x) \nabla f(x) + \int_{\R^d \setminus\{0\}} (f(x+y)-f(x)-\nabla f(x) \cdot y \I_{(0,1)}(|y|)) \, \mu(x,dy),
		\label{e.perturb}
	\end{equation}
	where the drift $b$ and the jumping kernels $\mu(x,dy)$ depend measurably on $x \in \R^d$. Assume that 
	\begin{equation*}
		\sup_{x \in \R^d} \left( |b(x)| + \int_{\R^d} \min\{1,|y|^{\beta}\} \, \mu(x,dy)\right)<\infty
	\end{equation*}
	for some $\beta \in [0,\rho)$ and that the compensated drift is zero in case that $\rho \leq 1$. Then 
\begin{enumerate}
[(a)]
\item The bp-closure of $A+\hat B$ is the full generator of a Markovian $\Cb$-Feller semigroup $S$, i.e.\ $S$ leaves $\Cb(\R^d)$ invariant and its restriction to that space is pointwise continuous. Moreover, $S$ enjoys the strong Feller property;
\item The $(A+\hat{B}, \Cc^\infty(\R^d))$-martingale problem is well-posed;
\item If additionally the tightness condition
	\begin{equation*}
		\lim_{R \to \infty} \sup_{x \in \R^d} \mu(x,\{y \in \R^d; |y| \geq R\})=0
	\end{equation*}
	holds, then $S$ is a $\Co$-semigroup, i.e.\ $S$ leaves $\Co(\R^d)$ invariant and its restriction to that space is strongly continuous. Thus $S$ is a Feller semigroup in this case.
	\end{enumerate}
\end{thm}

Theorem \ref{t.generic} applies to a wide class of Feller generators $A$ including strictly elliptic differential operators (see Section~\ref{s.diff}), generators of certain L\'evy processes (see Section~\ref{s.levy}) and operators of variable order (see Section~\ref{s.varorder}). Moreover, the well-posedness of the martingale problem in (b) yields uniqueness results for L\'evy-driven SDEs, see Section~\ref{s.sde}.  \medskip

This article is organized as follows. In Section~\ref{kernel}, we recall some notions concerning kernel operators, in Section~\ref{sg} those concerning semigroups and their generators. Here we also establish the connection between cores of of Feller generators and bp-cores of the corresponding full generators; this is important for our applications to martingale problems. In Section~\ref{pertFeller} we recall a perturbation result from \cite{K13} and prove our first main result (Theorem \ref{t.perturbFeller}). Section~\ref{resConv} concerns convergence results for perturbed semigroups, which is crucial for establishing sub Markovianity in Theorem \ref{t.generic}. Sections 2 -- 5 concern abstract state spaces $E$ rather than $\R^d$ and thus do not take the special structure of the operators $\cA$ in \eqref{eq.generalop} into account.

The remaining sections 6 -- 8 apply the abstract results to L\'evy-type operators as in \eqref{eq.generalop}. Section~\ref{perLTO} concerns the actual perturbation result and Section~\ref{mp} the corresponding martingale problem. The concluding Section~\ref{s.examples} contains examples and discusses applications in the theory of stochastic differential equations.

Moreover, there are two appendices which contain results that are used in the proof of Theorem \ref{t.perturbeuclidean} and which we believe to be of independent interest.

\section{Kernel operators} \label{kernel}

Throughout, $E$ is a locally compact Polish space. We endow $E$ with its Borel $\sigma$-algebra $\cB (E)$.
The spaces of bounded and measurable resp.\ bounded and continuous functions on $E$ are denoted by $\Bb(E)$ and $\Cb(E)$ respectively
whereas $\Co(E)$ refers to the space of continuous functions vanishing at $\infty$, i.e.\ of those continuous functions $f: E\to \R$ such that for every $\eps >0$ we find a compact set $K \subset E$ with $|f(x)|\leq \eps$ for all $x\in E\setminus K$.

A \emph{kernel} on $E$ is a map $k: E\times \cB(E) \to \mathds{R}$ such that 
\begin{enumerate}
\item the map $x\mapsto k(x,A)$ is measurable for every $A \in \cB(E)$;
\item  $k(x, \cdot)$ is a (signed) measure for every $x \in E$;
\item  $\sup_{x\in E} |k|(x, E) <\infty$, where $|k|(x, \cdot)$ denotes the total variation of $k(x, \cdot)$.
\end{enumerate}
If $k(x, \cdot)$ is a positive measure for every $x \in E$, then $k$ is called a \emph{positive kernel}; if every $k(x, \cdot)$ is  a (sub) probability measure, $k$ is a \emph{(sub) Markovian kernel}.

To every kernel $k$ on $E$, we can associate a bounded linear operator $T$ on $\Bb(E)$ by setting
\begin{equation}
	\label{eq.assop}
	(Tf)(x) \coloneqq \int_E f(y) \, k(x, dy), \quad f \in \Bb(E),\; x \in E.
\end{equation}
We call an operator $T$ of this form  a \emph{kernel operator} on $\Bb(E)$. It turns out that a bounded, linear operator on $\Bb(E)$ is a kernel operator if, and only if, $T$ is continuous with respect to the weak topology $\sigma \coloneqq \sigma (\Bb(E), \cM_b (E))$ induced by the space $\cM_b(E)$ of bounded signed measures on $E$ (see e.g.\ \cite[Prop.\ 3.5]{k11}). For sequences, $\sigma$-convergence is nothing else than \emph{bp-convergence} (bp is short for \emph{b}ounded and \emph{p}ointwise), i.e.\ $f_n \to f$ with respect to $\sigma$ if, and only if, $\sup_{n \in \N} \|f_n\|_{\infty}<\infty$ and $f_n \to f$ pointwise. Indeed, by dominated convergence, bp-convergence implies $\sigma$-convergence; the converse follows from the uniform boundedness principle. 
We will write $\cL(\Bb(E), \sigma)$ for the space of $\sigma$-continuous linear operators on $E$, i.e.\ kernel operators. Note that any such operator is automatically bounded. In what follows, we write $\weak$ to indicate convergence with respect to $\sigma$ while we reserve $\to$ to indicate convergence with respect to the supremum norm.
If $T$ is a bounded operator defined via \eqref{eq.assop} on $\Cb(E)$ or $\Co(E)$, then $T$ is also called a kernel operator (on $\Cb(E)$ resp.\ $\Co(E)$); it can be extended in a unique way to a kernel operator on $\Bb(E)$. As is well known, every bounded operator on $\Co(E)$ is a kernel operator and can thus be extended to a kernel operator on $\Bb(E)$, see e.g.\ \cite[Cor.\ 21.12]{mims}.

In applications, it is often of interest if a kernel operator on $\Bb(E)$ leaves one of the spaces $\Cb(E)$ or $\Co(E)$ invariant. Of particular interest is the case where a kernel operator $T$ maps $\Bb(E)$ into $\Cb(E)$. Such an operator is called \emph{strong Feller operator}. For our perturbation results, it will be important to know under which conditions a kernel operator leaves the space
\[
\Bo(E) \coloneqq \big\{ f\in \Bb(E): \forall \eps >0\, \exists\, K\Subset E : |f(x)| \leq \eps\, \forall \, x\in E\setminus K\big\}.
\]
invariant. We present two results on this topic.

\begin{lem}\label{l.b0}
	Let $T$ be a kernel operator with associated kernel $k$. Then $T\Bo(E)\subset \Bo(E)$ if and only if $k(\cdot, K) \in \Bo(E)$ for every relatively compact sets $K$. 
\end{lem}

\begin{proof}
As $T\one_K = k(\cdot, K)$ and $\one_K \in \Bo(E)$ for a  relatively compact set $K$, the stated condition is certainly necessary. To prove its sufficiency, consider a function $f\in \Bo(E)$ with compact support $S$. We find a sequence of simple functions $\tilde f_n$ that converges uniformly to $f$. If we set $f_n \coloneqq \tilde f_n \one_S$, then also the sequence $f_n$ converges uniformly to $f$ and, moreover, for $c\neq 0$ and every $n\in \N$ the set $\{f_n = c\}$ is relatively compact.
Thus, $Tf_n\in \Bo(E)$. As $Tf_n \to Tf$ uniformly and since $\Bo(E)$ is closed with respect to the supremum norm, $Tf\in \Bo(E)$. The case of a general function $f\in \Bo(E)$ follows from this by approximating $f$ uniformly with a sequence of measurable functions with compact support. 
\end{proof}

If the operator $T$ is positive, then invariance of $\Co(E)$ implies invariance of $\Bo(E)$.

\begin{lem}\label{l.c0}
	Let $T$ be a positive kernel operator with $T\Co(E)\subset \Co(E)$. Then we have $T\Bo(E) \subset \Bo(E)$. If, in addition, $T$ has the strong Feller property, then $T\Bo(E) \subset \Co(E)$. 
\end{lem}

\begin{proof}
Let $0\leq f\in \Bo(E)$. Given $\eps>0$, pick a compact set $K$ such that $|f(x)|\leq \eps/\|T\|$ for $x\in E\setminus K$. Then
\[
0\leq f\leq \|f\|_\infty \one_K + \frac{\eps}{\|T\|}\one_{E\setminus K}
\]
and, consequently,
\[
0\leq Tf \leq \|f\|_\infty k(\cdot, K) + \eps.
\]
Since $T\Co(E)\subset T\Co(E)$, \cite[Lem.\ 3.2.15]{J05} yields $k(\cdot, K) \in \Bo(E)$ and thus $Tf\in \Bo(E)$ follows as $\eps>0$ was arbitrary. The addendum follows from the identity $\Co(E) = \Bo(E)\cap \Cb(E)$.
\end{proof}

\section{Semigroups, (pseudo)resolvents and generators} \label{sg}

We now turn our attention to semigroups of kernel operators. A family $T= (T(t))_{t>0} \subset \cL (\Bb(E), \sigma)$ is called \emph{semigroup of kernel operators} if
\begin{enumerate}
\item $T(t+s) = T(t)T(s)$ for all $s,t>0$.
\item $T$ is \emph{exponentially bounded}, i.e.\ there are constants $\omega\in \R$, $M>0$ such that $\|T(t)\|\leq Me^{\omega t}$. We say that $T$ is \emph{of type $(M, \omega)$} to stress these constants.
\item The map $(t,x) \mapsto(T(t)f)(x)$ is measurable.
\end{enumerate}
It is not difficult to see that if $T$ is a semigroup of kernel operators of type $(M,\omega)$ then for every $\lambda\in \C$ with $\Re\lambda >0$ there is an operator $R(\lambda) \in \cL(\Bb(E), \sigma)$ such that
\begin{equation}
\label{eq.resolvent}
\la R(\lambda) f, \mu \ra = \int_0^\infty e^{-\lambda t}\la T(t)f, \mu\ra\, dt
\end{equation}
for all $f\in \Bb(E)$ and $\mu \in \cM_b(E)$.

In the terms of \cite[Def.\ 5.1]{k11}, a semigroup of kernel operators is an integrable semigroup on the norming dual pair $(\Bb(E), \cM_b(E))$. It turns out that the family $(R(\lambda))_{\Re \lambda>\omega}$ is a \emph{pseudo-resolvent}, i.e.\  it satisfies the resolvent identity $R(\lambda)-R(\mu)=(\mu-\lambda)R(\lambda)R(\mu)$, cf.\  \cite[Prop. 5.2]{k11}. For more details on pseudoresolvents we refer to \cite[Sect. III.4.a]{en}.

In general, the family $(R(\lambda))_{\Re\lambda >\omega}$ does \emph{not} consist of injective operators, and so it is not the resolvent of a (single-valued) operator. However, there is a \emph{multivalued} operator $\hat A$ such that $R(\lambda) = (\lambda - \hat A)^{-1}$ for $\Re\lambda >\omega$ (see \cite[Appendix A]{haase} for more information concerning multivalued operators). We call $\hat{A}$ the \emph{full generator} of the semigroup $T$. The full generator can be characterized equivalently as follows:
\begin{equation*}
	(f,g) \in \hat{A} \iff \forall t>0\::\: T(t) f-f = \int_0^t T_s g \, ds, 
\end{equation*}
cf.\ \cite[Prop. 5.7]{k11}. In particular, our terminology is consistent with that used by Ethier--Kurtz \cite[Sect.\ 1.5]{ek}.\smallskip

We now introduce additional properties that a semigroup of kernel operators might have. If $T$ is of type $(1,0)$, then $T$ is called a \emph{contraction} semigroup of kernel operators. If $T$ is a contraction semigroup of kernel operators and every operator $T(t)$ is positive, we say that $T$ is \emph{sub Markovian}. If additionally $T(t)\one = \one$ for all $t>0$, then $T$ is called \emph{Markovian}. 

We will call $T$ a \emph{$\Co$-semigroup} if $T(t)\Co(E)\subset \Co(E)$ and the restriction of $T$ to $\Co(E)$ is \emph{strongly continuous}, i.e.\ for every $f\in \Co(E)$ we have $T(t)f \to f$ with respect to $\|\cdot\|_\infty$ as $t\to 0$. A $\Co$-semigroup that is also sub Markovian is called a \emph{Feller semigroup}.

We will call $T$ a \emph{$\Cb$-semigroup} if $T(t)\Cb(E)\subset \Cb(E)$ and the restriction of $T$ to $\Cb(E)$ is 
\emph{stochastically continuous}, i.e.\ $T(t)f\weak f$ as $t\to 0$ for $f\in \Cb(E)$. We note that, given sub Markovianity, this continuity condition can be equivalently characterized in terms of the associated kernels $p_t$ by asking that $p_t(x, B(x, \eps)) \to 1$ as $t\to 0$ for every $x\in E$ and $\eps>0$, see \cite[Lem.\ 3.2.17]{J05}. A $\Cb$-semigroup that is also sub Markovian is called \emph{$\Cb$-Feller semigroup}.\smallskip

Let us point out that the above-defined objects are compatible with the classical notions. To wit: If $T$ is a $\Co$-semigroup, then the operator $R(\lambda)$ leaves the space  $\Co(E)$ invariant for $\Re \lambda >\omega$ and its restriction to that space is injective. Thus, the restriction of $(R(\lambda))_{\Re \lambda>\omega}$ is the resolvent of a (single valued) operator $A$. Moreover, $A$ is the \emph{part}
of $\hat A$ in $\Co(E)$, i.e.
\[
D(A) = \{ u \in \Co(E) : \exists f\in \Co(E) \mbox{ s.t.\ } (u,f) \in \hat A\},\qquad Au = f.
\]
Identifying $A$ with its graph, we have $A= \hat A\cap (\Co(E)\times \Co(E))$. By general semigroup theory, $A$ coincides with the generator of the semigroup $T|_{\Co(E)}$, defined as (norm-)derivative in $0$.

Conversely, if we start with a strongly continuous semigroup on $\Co(E)$, we may extend it to $\Bb(E)$ as any bounded linear operator on $\Co(E)$ is a kernel operator. In fact, the extended semigroup is a semigroup of kernel operators in the sense defined above (this follows from \cite[Lem.\ 6.1]{k11}). If the semigroup is sub Markovian, then the extension is also stochastically continuous, see \cite[Lem.\ 4.8.7]{J01}.\smallskip

We next discuss how to recover the full generator $\hat A$ from the operator $A$. To that end, we make use of operator cores. We recall that a subset $D\subset D(A)$ is called a \emph{core} for $A$ if for every $u\in D(A)$ we find a sequence $(u_n)\subset D$ with 
$u_n \to u$ and $Au_n \to Au$. To use cores in the context of multivalued operators, we employ the notion of bp-convergence defined in Section~\ref{kernel}. A set $M \subset \Bb(E)$ is called \emph{bp-closed} if with every bp-convergent sequence, it also contains its limit. The \emph{bp-closure} of a set $M$ is the smallest bp-closed subset of $\Bb(E)$ that contains $M$. We use these notions, mutatis mutandis, also in $\Bb(E)\times \Bb(E)$. We call a subset $M \subset \hat A$ a \emph{bp-core} if the bp-closure of $M$ equals $\hat A$.

\begin{lem}\label{l.bp}
Let $T$ be $\Co$-semigroup with full generator $\hat A$ and let $A$ be the part of $\hat A$ in $\Co(E)$. Let $D$ be a core for $A$. 
Then  $\{(u, Au) : u \in D\}$ is a bp-core for $\hat A$.
\end{lem}

\begin{proof}
Let $T$ be of type $(M, \omega)$ and fix $\lambda >\omega$. As $D$ is a core for $A$,  the set $(\lambda - A)D$ is $\|\cdot\|$-dense in $\Co(E)$. It follows that the bp-closure of $(\lambda - A)D$ contains $\Co(E)$ and thus equals all of $\Bb(E)$ (see Proposition \ref{p.bpdense}).

Put $S\coloneqq \{ (u, \lambda u -Au) : u\in D\}$.
We note that $\hat A$, and hence also $\lambda - \hat A$, is $\sigma$-closed and thus, in particular, bp-closed. As $S\subset\lambda- \hat A$ it follows that the bp-closure $\bar S$ is a subset of $\lambda - \hat A$. Now define $C\coloneqq \{ f\in \Bb(E) : \exists\,  u \mbox{ s.t.\ } (u, f) \in \bar S\}$.
Then $C$ is bp-closed.

Indeed, Let $(f_n) \subset C$ be bp-convergent to $f$. Pick $u_n$ such that $(u_n, f_n ) \in \bar S$. As $R(\lambda) = (\lambda - \hat A)^{-1}$, 
we see that $(R(\lambda)f_n, f_n) \in \hat B$. 
It follows that $u_n - R(\lambda) f_n \in \mathrm{ker}(\lambda -\hat A) = \{0\}$, i.e.\ $u_n = R(\lambda) f_n$.
By the $\sigma$-continuity of $R(\lambda)$, we have that $u_n \weak u\coloneqq R(\lambda)f$. This implies that
$(u,f) \in \bar S$ whence $f\in C$, proving that $C$ is bp-closed as claimed.

As $C$ contains $(\lambda - A)D$, it follows that $C= \Bb(E)$. But this entails that $\bar S= \lambda -\hat A$. We have proved that $\lambda -\hat A$ is the bp-closure of $S$ which is equivalent to the claim. 
\end{proof}

\section{Perturbation of (strong) Feller semigroups} \label{pertFeller}

We begin by recalling a result from \cite{K13} concerning the perturbation of semigroups of kernel operators consisting of strong Feller operators. We will use the following set of assumptions.

\begin{hyp}\label{h1}
Let $T$ be a $\Cb$-semigroup of type $(M, \omega)$  that consists of strong Feller operators. We denote by $\hat A$ the full generator of the semigroup and write $(R(\lambda))_{\Re\lambda>\omega}$ for the Laplace transform of $T$. Moreover, let $\hat B : D(\hat B) \to \Bb(E)$ be a single-valued linear operator with $D(\hat A) \subset D(\hat B)$
which satisfies the following assumptions:
\begin{enumerate}
\item For $t>0$ the operator $\hat B T(t)$, initially defined on $D(\hat A)$, has an extension to an operator in $\cL (\Bb(E), \sigma)$ (which, by slight abuse of notation, we still denote by $\hat BT(t)$);
\item $\hat BR(\lambda) \in \cL(\Bb(E),\sigma)$ for one/all $\lambda >\omega$;
\item the function $(t,x) \mapsto \hat B T(t)f(x)$ is measurable for $f\in \Bb(E)$;
\item there is a function $\varphi$ which is integrable in a neighborhood of $0$ such that $\|\hat B T(t)\|\leq \varphi(t)$ for $t>0$.
\end{enumerate}
\end{hyp}

\begin{thm}{{(\cite[Thm.\ 3.3]{K13})}} \label{t.k13}
Assuming Hypothesis \ref{h1} the operator $\hat A + \hat B$ (defined on $D(\hat A) \subset D(\hat B)$) is the full generator of a $\Cb$-semigroup $S$ that consists of strong Feller operators. This semigroup satisfies for $t>0$ the \emph{Duhamel formula}
\begin{equation}
\label{eq.formula}
S(t)f = T(t)f + \int_0^t S(t-s)\hat B T(s)f\, ds.
\end{equation}
Moreover, we can develop $S$ in its \emph{Dyson--Phillips series}:
\begin{equation}
\label{eq.dp}
S(t) = \sum_{n=0}^\infty S_n(t)\quad \mbox{where}\quad S_0(t) = T(t)\quad\mbox{and}\quad S_{n+1}(t)f = \int_0^t S_n(t-s)\hat B T(s)\, ds.
\end{equation}
Here, all integrals have to be understood in the weak $\cM_b(E)$-sense as in \eqref{eq.resolvent}.
\end{thm}

It is a rather obvious question whether the new semigroup $S$ is a $\Co$-semigroup if this is the case for $T$. Unfortunately, without further assumptions on $\hat B$ this is not the case.

\begin{example}\label{ex.noc0}
For $E= \R$, consider the heat semigroup $T$, given by
\[
(T(t)f)(x) \coloneqq \frac{1}{\sqrt{2\pi t}} \int_\R \exp \left(-\frac{(x-y)^2}{4t} \right) f(y)\, dy\quad (x\in \R)
\]
for $f\in \Bb(\R)$. It is well known that $T$ is a $\Co$-semigroup which also consists of strong Feller operators. Now consider the operator $\hat B: f\mapsto f(0) \cdot \one_\R$. Then $\hat B \in \cL (\Bb(\R), \sigma)$ (whence it satisfies Hypothesis \ref{h1}, see
\cite[Ex.\ 3.4]{K13}). It follows that $\hat B T(s)f = c\cdot\one_\R$, for some constant $c=c(s,f)$. Consequently, $T(t-s)\hat B T(s)f= c$ for all $t\geq s$. 
If $0<f \in \Co(\R)$ then $c_t(f)\coloneqq \inf\{c(s,f) : 0\leq s \leq t\}>0$ for all $t>0$ and we can infer from the Dyson--Phillips expansion \eqref{eq.dp} that the perturbed semigroup $S$ does not leave $\Co(\R)$ invariant. Indeed, fixing $t>0$, we see that $S(t)f \geq S_1(t)f \geq c_t(f)\one_\R$ and the latter does not vanish at infinity.
\end{example}

\begin{thm}\label{t.perturbFeller}
Assume in addition to  Hypothesis \ref{h1} that $T$ is a positive $\Co$-semigroup. 
\begin{enumerate}
[(a)]
\item If $\hat D$ is a bp-core for $\hat A$, then $\{ (u, f+ \hat B u) : (u,f) \in \hat D\}$ is a bp-core for $\hat A + \hat B$.
\item If $\hat BT(t) \Bo(E) \subset \Bo(E)$ for every $t>0$, then $S$ is a $\Co$-semigroup. 
\end{enumerate}
\end{thm}

\begin{proof}
(a) Let $\hat D$ be a bp-core for $\hat A$ and denote by $\bar S$ the bp-closure of $\{(u, f+ \hat B u), (u, f) \in \hat D\}$. As $\hat A + \hat B$ is a full generator, it is $\sigma$-closed and thus, in particular, bp-closed and it follows that $\bar S\subset \hat A +\hat B$. Note that this entails that for $(u,g) \in \bar S$ we have $u \in D(\hat A)$. 
We now define $W\coloneqq \{(u, g- \hat B u) : (u, g) \in \bar S\}$. Clearly, $\hat D \subset W \subset \hat A$. We prove that $W$ is bp-closed. 
To that end, let $(u_n, g_n-\hat B u_n)$ be a sequence in $W$ (thus $(u_n, g_n) \in \bar S$) such that $u_n$ bp-converges to $u$ and 
$f_n \coloneqq g_n - \hat B u_n$ bp-converges to $f$. Then $(u_n, f_n) \subset \bar A$ bp-converges to $(u, f)$. By the bp-closedness of $\hat A$, we have $(u,f)\in \hat A$. Picking $\lambda >\omega$, this is equivalent to $u= R(\lambda ) (\lambda u- f)$ and the same equality holds with $u/f$ replaced by $u_n/f_n$. Using that $\hat B R(\lambda)$ is $\sigma$-continuous by (ii) in Hypothesis \ref{h1}, we infer
\[
\hat B u_n = \hat B R(\lambda)(\lambda u_n - f_n) \weak \hat B R(\lambda) (\lambda u - f) = \hat B u.
\]
Setting $g\coloneqq f+\hat B u$, the sequence $(u_n, g_n) \subset \bar S$ bp-converges to $(u,g)$. As $\bar S$ is bp-closed, $(u,g) \in \bar S$ and thus $(u, f) \in W$. Since $\hat D$ is a bp-core for $\hat A$, it follows that $W= \hat A$. But this is equivalent to $\bar S = \hat A + \hat B$ which means that $\{(u, f+ \hat B u), (u, f) \in \hat D\}$ is indeed a bp-core for $\hat A + \hat B$.\medskip

(b) We first prove by induction that the operators $S_n(t)$ appearing in the Dyson--Phillips series \eqref{eq.dp} map $\Bo(E)$ to $\Co(E)$. In particular, they leave the space $\Bo(E)$ invariant.

For $n=0$ this follows from Lemma \ref{l.c0} and our assumption that $T$ is a positive $\Co$-semigroup that enjoys the strong Feller property.

Let us assume that, for some $n\in \N$, we have already proved that $S_n(t)$ maps $\Bo(E)$ to $\Co(E)$  for every $t>0$. Fix $f\in \Bo(E)$ and $t>0$. By our assumption, we have $\hat BT(s)f \in \Bo(E)$ for every $s\in (0,t)$. By induction hypothesis, $S_n(t-s)$ maps $\Bo(E)$ to $\Co(E)$, whence $S_n(t-s)\hat B T(s) f \in \Co(E)$ for $s\in (0,t)$. Since $\Co(E)$ is separable with dual space $\cM_b (E)$, the Pettis measurability theorem implies that the function $s\mapsto S_n(t-s)\hat BT(s)f$ is Bochner integrable which, in turn, implies that
$S_{n+1}(t)f = \int_0^t S_n(t-s)\hat B T(s)f\, ds \in \Co(E)$.\smallskip

Note that the series in \eqref{eq.dp} converges for small $t$ even in the operator norm. From this it follows that for small $t$ we have $S(t)f \in \Co(E)$. As $f$ was arbitrary, $S(t)\Co(E)\subset \Co(E)$ for small $t$ and hence, by the semigroup law, for all $t>0$. The strong continuity of the restriction of $S$ to $\Co(E)$ follows from that of $T$, the identity \eqref{eq.formula} and the observation that the integral in that formula converges to 0 in operator norm as $t\mapsto 0$. 
\end{proof}

\section{Resolvent convergence of the perturbed operators} \label{resConv}

In this section, we will consider a sequence $\hat{B}_n$ of perturbations that converges, in a certain sense, to the operator $\hat{B}$. We want to know under which assumptions and in which sense the pseudoresolvents $(\lambda - (\hat A + \hat{B}_n))^{-1}$ converge to 
$(\lambda - (\hat A + \hat B))^{-1}$ as $n\to \infty$. Besides being interesting in its own right, such a convergence result will also allow us to establish dissipativity and/or resolvent positivity for large classes of operators with measurable coefficients.

The key to prove our convergence result is the following Lemma, which is taken from \cite[Lem. 3.8]{K13}.

\begin{lem}\label{l.perturbedresolvent}
Assuming Hypothesis \ref{h1}, we have $\|\hat B R(\lambda)\|\to 0$ as $\lambda \to \infty$. Moreover, if $\lambda$ is so large that
$\|\hat BR(\lambda) \| <1$, then 
\[
(\lambda - (\hat A + \hat B))^{-1} = R(\lambda) \sum_{k=0}^\infty (\hat B R(\lambda))^k,
\]
where the latter series converges in operator norm.
\end{lem}

\begin{prop}\label{p.resolventconvergence}
Assume Hypothesis \ref{h1} and let $(\hat B_n, D(\hat B_n))$ be an operator such that Hypothesis \ref{h1} is also fulfilled with $\hat B$ replaced by $\hat B_n$ (but the same semigroup $T$). Moreover assume that
\begin{enumerate}
\item $\sup_{n\in \N}\|\hat B_n R(\lambda)\|\to 0$ as $\lambda \to \infty$,
\item for every $h\in \Bb(E)$ and large enough $\lambda$ we have $\hat B_n R(\lambda) h \weak \hat BR(\lambda)h$, and
\item whenever $(h_n)\subset \Bb(E)$ satisfies $h_n \weak 0$, we have $\hat B_nR(\lambda)h_n \weak 0$ for large enough $\lambda$.
\end{enumerate}
In this case, for large enough $\lambda$ we have
\begin{equation}\label{eq.approxtoprove}
(\lambda - (\hat A + \hat B_n ))^{-1} f \weak (\lambda - (\hat A + \hat B))^{-1}f
\end{equation}
for all $f\in \Bb(E)$.
\end{prop}

\begin{proof}
By assumption (i), we can pick $\lambda_0$ so large that for $\lambda \geq \lambda_0$ we have $\|\hat B_nR(\lambda)\|\leq 1/2$ for every $n\in \N$ and $\lambda\geq \lambda_0$. By Lemma \ref{l.perturbedresolvent}, we have
\begin{equation}\label{eq.approxresolvent}
(\lambda - (\hat A + \hat B_n))^{-1} = R(\lambda)\sum_{k=0}^\infty (\hat B_n R(\lambda))^k.
\end{equation}
Now let $f\in \Bb(E)$ and $\eps>0$. We fix $N\in \N$ such that $\sum_{k=N+1}^\infty \|(\hat B_n R(\lambda))^kf\| \leq \eps$ for all
$n\in \N$. 

We prove that, for every $k\in \N$, we have $(\hat B_n R(\lambda))^kf \weak (\hat BR(\lambda))^kf$ as $n\to \infty$.

For $k=0$, there is nothing to prove and for $k=1$ this is exactly our assumption (ii). Assume that we know this convergence to be true for some $k\in \N$. We set $g_n \coloneqq (\hat B_n R(\lambda))^k f$ and $g\coloneqq (\hat B R(\lambda))^k f$. Then
\[
(\hat B_n R(\lambda))^{k+1}f - (\hat B R(\lambda))^{k+1}f = \hat B_n R(\lambda)(g_n-g) + (\hat B_nR(\lambda)g - \hat BR(\lambda)g)
\weak 0.
\]
Indeed, by induction hypothesis $g_n-g\weak 0$ so that the first term converges weakly to $0$ by our assumption (iii) and that the second term converges weakly to $0$ follows from the case $k=1$.

Alltogether, we find that $\sum_{k=0}^\infty(\hat B_n R(\lambda))^k f \weak \sum_{k=0}^\infty (\hat B R(\lambda))^kf$. Given that the operator $R(\lambda)$ is $\sigma$-continuous, the representation \eqref{eq.approxresolvent} and Lemma \ref{l.perturbedresolvent} yield
\eqref{eq.approxtoprove}
\end{proof}

If the operator $R(\lambda)$ is positive (which is always the case in our main application) we obtain stronger convergence results.

\begin{cor}\label{c.strongerconvergence}
In the situation of Proposition \ref{p.resolventconvergence}, assume additionally that $R(\lambda)\geq 0$ for large enough $\lambda$. Then the convergence in \eqref{eq.approxtoprove} is not only with respect to $\sigma$, but even uniform on compact subsets of $E$.
\end{cor}

\begin{proof}
Fix $f\in \Bb(E)$ and put $g_n\coloneqq \sum_{k=0}^\infty (\hat B_n R(\lambda))^kf$ and $g \coloneqq \sum_{k=0}^\infty (\hat B R(\lambda))^kf$. Fixing $m\in \N$, we put $h_m \coloneqq \sup_{n\geq m} |g_n-g|$. We have seen in the proof of Proposition \ref{p.resolventconvergence} that $g_n \weak g$ and, consequently, $h_m \downarrow 0$ pointwise.

As $R(\lambda)$ is positive, we find for $n\geq m$
\begin{align*}
\big| (\lambda - (\hat A+\hat B_n))^{-1}f - (\lambda - (\hat A +\hat B))^{-1}f\big| & = |R(\lambda)g_n - R(\lambda)g|\\
& \leq R(\lambda)|g_n-g|\\
& \leq R(\lambda) h_m.
\end{align*}
The latter converges to $0$ as $m$, hence $n$, tends to $\infty$. However, as $R(\lambda)$ is positive this is even a monotone convergence and thus the convergence is uniform on compact subsets of $E$ by Dini's theorem.
\end{proof}

It is natural to ask whether the perturbed semigroups also converge. In general, weak convergence of the resolvents does not imply weak convergence of the corresponding semigroups, see \cite{eisner} for a concrete example. In Section~\ref{mp}, we will use the theory of martingale problems to establish a convergence result for semigroups associated with L\'evy-type operators.

We can also prove that certain features of the pseudoresolvents are stable under the convergence described in Proposition \ref{p.resolventconvergence}.

\begin{cor}\label{c.stability}
In the situation of Proposition \ref{p.resolventconvergence}, the following holds true.
\begin{enumerate}
\item\label{c.stability.i} Let $\lambda \in \C$ with $\Re \lambda>0$ be in the resolvent set of $\hat{A}+\hat{B}_n$ and assume that $\|(\lambda - (\hat A + \hat B_n))^{-1}\| \leq (\Re\lambda)^{-1}$ for all $n\in \N$. Then $\lambda$ is in the resolvent set of $\hat A + \hat B$ and
$\|(\lambda - (\hat A + \hat B))^{-1}\| \leq (\Re\lambda)^{-1}$. In this case, the semigroup $S$, generated by $\hat A + \hat B$, is contractive.
\item\label{c.stability.ii} Suppose that for $\lambda \geq \lambda_0$ we have $(\lambda + (\hat A + \hat B_n))^{-1}\geq 0$ for all $n\in \N$. Then also
$(\lambda + (\hat A + \hat B))^{-1} \geq 0$ for $\lambda \geq \lambda_0$. In this case, the semigroup $S$ generated by $\hat A + \hat B$ is positive.
\end{enumerate}
\end{cor}

\begin{proof}
We put $F_n(\lambda) \coloneqq (\lambda - (\hat A + \hat B_n))^{-1}$ and $F(\lambda) \coloneqq (\lambda - (\hat A + \hat B))^{-1}$
whenever these are defined.

\ref{c.stability.i} The sequence $F_n$ is an (operator-valued) holomorphic function that is, by assumption, locally bounded on $\{\Re \lambda > 0\}$.
Fixing $f\in \Bb(E)$ and $\mu \in \cM_b(E)$, the same is true for the scalar function $\varphi_n: \lambda \mapsto  \varphi_n(\lambda) \coloneqq \langle F_n(\lambda)f, \mu\rangle$. By Proposition \ref{p.resolventconvergence}, for real $\lambda$ that are large enough, $\varphi_n(\lambda)$ converges to $\varphi(\lambda) \coloneqq\langle F(\lambda)f, \mu\rangle$. It follows from Vitali's theorem (see \cite[Thm.\ 2.1]{an00}) that $\varphi$ has a holomorphic extension to the set $\{\Re\lambda >0\}$ and $\varphi_n$ converges locally uniformly to $\varphi$. As $f$ and $\mu$ are arbitrary, it follows from a vector-valued analytic extension theorem \cite[Thm.\ 3.5]{an00} that $F$ has an analytic continuation to $\{\Re  \lambda  >0\}$ which proves $\{\Re\lambda >0\}$ is contained in the resolvent set of $\hat{A}+\hat{B}$. Moreover
\[
|\langle F(\lambda) f, \mu \rangle| = \lim_{n\to \infty} |\langle F_n(\lambda) f, \mu \rangle| \leq \limsup_{n\to \infty} (\Re\lambda)^{-1}\|f\|\|\mu\|
\]
for every $\Re\lambda >0$. This implies the estimate for the Laplace transform. As for the contractivity of the semigroup, we note that for
$f\in \Cb(E)$ and $\cM_b(E)$ the orbit $t\mapsto \langle S(t)f, \mu \rangle$ is continuous whence the (scalar) Post--Widder inversion formula (see \cite[Thm.\ 1.7.7]{abhn}) yields
\[
\langle S(t)f, \mu \rangle = \lim_{n\to\infty} \Big\langle \Big(\frac{n}{t}\Big)^n \Big(\frac{n}{t} - \hat A - \hat B\Big)^{-n}f, \mu \Big\rangle.
\]
By the resolvent estimate, the absolute value of the right-hand side is at most $\|f\|\|\mu\|$ and it follows that
$S$ is contractive on $\Cb(E)$. But then it is also contractive on $\Bb(E)$.
\smallskip

\ref{c.stability.ii} For $0\leq f\in \Bb(E)$, $0\leq \mu \in \mathcal{M}(E)$ and $\lambda \geq 0$, we have
\[
\langle F(\lambda)f, \mu\rangle  = \lim_{n\to \infty} \langle F_n(\lambda) f, \mu \rangle \geq 0
\]
as for $\lambda \geq \lambda_0$ we have $F_n(\lambda) \geq 0$ for all $n\in N$. As $f$ and $\mu$ were arbitrary, this proves $F(\lambda)\geq 0$ for $\lambda \geq \lambda_0$. As above, we can infer positivity of the semigroup $S$ from this by means of the Post--Widder inversion formula.
\end{proof}

\section{Perturbation of strong Feller semigroups by L\'evy-type operators} \label{perLTO}

In this section, we work in the Euclidean setting, $E=\R^d$, and consider perturbations of strong Feller semigroups by a class of integro-differential operators. To that end, we will make more concrete assumptions on our initial semigroup $T$ which are tailor-made for this situation.

\begin{hyp}\label{h1b}
Assume that $T=(T(t))_{t\geq 0}$ is a $\Cb$-semigroup of type $(M, \omega)$ with state space $E=\R^d$ that consists of strong Feller operators. Moreover, the following conditions hold for some $\rho>0$:
\begin{enumerate}
[(i)]
\item $T(t)f \in \Cb^\rho(\R^d)$ for all $t>0$ and $f\in \Bb(\R^d)$;
\item There is a function $\varphi \in L^1(0,1)$ such that
\[
\|T(t)f\|_{\Cb^\rho} \leq \varphi (t)\|f\|_\infty
\]
for all $t \in (0,1)$ and $f\in \Bb(\R^d)$;
\end{enumerate}
\end{hyp}

As before, we will denote the Laplace transform of our semigroup $T$ by $(R(\lambda))_{\Re\lambda >\omega}$ and the full generator of $T$ by $\hat A$. For future reference, we note the following consequence of Hypothesis \ref{h1b}.

\begin{lem}\label{l.hoelderdomain}
Assume Hypothesis \ref{h1b}. Then, for every $\lambda \in \C$ with $\Re \lambda>\omega$, we have $R(\lambda )f \in \Cb^{\rho}(\R^d)$ for all $f\in \Bb(\R^d)$ and 
\[
\|R(\lambda)f\|_{\Cb^{\rho}} \leq C(\lambda) \|f\|_\infty, \quad f\in \Bb(\R^d),
\]
for a constant $C(\lambda)$ with $C(\lambda) \to 0$ as $\Re\lambda \to \infty$. In particular, $D(\hat A) \subset \Cb^{\rho}(\R^d).$
\end{lem} 

\begin{proof}
	For $r>0$ it follows from the exponential boundedness of the semigroup and Hypothesis~\ref{h1b}(ii)  that
	\[
	\|T(1+r)f\|_{\Cb^\rho} \leq \|T(1)\|_{\cL(\Cb^\rho, \Bb)} \|T(r)\| \, \|f\|_\infty \leq M\varphi(1)e^{\omega r} \|f\|_\infty.
	\]
	Setting $\psi(t) = \varphi (t)$ for $t\in (0,1)$ and  $\psi (t) = M\varphi(1) e^{\omega (t-1)}$ for $t \geq 1$, we get
	\begin{equation}
		\|T(t) f\|_{\Cb^{\rho}} \leq \psi(t) \|f\|_{\infty}, \quad t >0, \; f \in \Bb(\R^d).
		\label{eq-reg.sg.alltimes}
	\end{equation}
	Now fix $f \in \Bb(\R^d)$ and $x \in \R^d$. Since $\|T(t)\| \leq M e^{\omega t}$ and $t \mapsto T(t) f(x)$ is Borel measurable, the integral $R(\lambda) f(x) = \int_{(0,\infty)} e^{-\lambda t} T(t) f(x) \, dt$ exists as a Lebesgue integral for any $\lambda$ with $\Re \lambda>\omega$. If $\rho \in (0,1)$, then it is immediate from \eqref{eq-reg.sg.alltimes} and the estimate
	\begin{equation*}
		|R(\lambda)f(x) - R(\lambda)f(z)|
		\leq \int_0^\infty e^{-\lambda t} |T(t)f(x)-T(t)f(z)| \, dt
	\end{equation*}
	that $\|R(\lambda) f\|_{\Cb^{\rho}} \leq C(\lambda) \|f\|_{\infty}$, where $C(\lambda):= \int_0^\infty \psi(t) e^{-\lambda t} \, dt$. If $\rho \geq 1$, then the differentiation lemma for parameter-depend integrals, see e.g.\ \cite[Thm.\ 12.5]{mims} or \cite[Prop.\ A.1]{euler-maruyama}, implies that $x \mapsto R(\lambda) f(x)$ is differentiable and
	\begin{equation*}
		\frac{\partial}{\partial x_i} R(\lambda) f(x) = \int_0^\infty e^{-\lambda t} \frac{\partial}{\partial x_i} T(t)f(x) \, dt
	\end{equation*}
	for all $\Re \lambda>\omega$. Thus, by \eqref{eq-reg.sg.alltimes}, $\|R(\lambda) f\|_{\Cb^{\rho}} \leq C(\lambda) \|f\|_{\infty}$ with $C(\lambda)$ as before. By dominated convergence, $C(\lambda) \to 0$ as $\Re \lambda \to \infty$.
\end{proof}

We now introduce the integro-differential operator $\hat B$ that we will consider as a perturbation. We fix a function $\chi$ such that $\one_{B(0,1)}\leq \chi\leq \one_{B(0,2)}$ and put
\begin{equation}
	\hat{B}f(x) = b(x) \cdot \nabla f(x) + \int_{\R^d \setminus\{0\}} (f(x+y)-f(x)-y \cdot \nabla f(x) \chi(y)) \, \mu(x,dy).
	\label{eq-intdiff}
\end{equation}

We make the following standing assumption.

\begin{hyp}\label{h2} 
The function $b: \R^d \to \R^d$ and the kernel $\mu : \R^d\times \cB(\R^d\setminus\{0\}) \to [0, \infty]$ satisfy

\begin{enumerate}
[(i)]
	\item $b$ is Borel measurable and bounded;
	\item $x \mapsto \int_{\R^d \setminus \{0\}} f(y) \, \mu(x,dy)$ is Borel measurable for every $f \in \Cc(\R^d \setminus \{0\})$. 
	\item There is a constant $\beta \in (0,2)$ such that
\begin{equation}
	\|\mu\|_{\beta} := \sup_{x \in \R^d} \left( \int_{\R^d \setminus \{0\}} \min\{|y|^{\beta},1\} \, \mu(x,dy)  \right)<\infty.
	\label{eq-kernel.integrability}
\end{equation}
Moreover, $\beta$ is strictly smaller than the constant $\rho$ from Hypothesis~\ref{h1b}.
\item If $\rho \leq 1$, then the \emph{compensated drift} $b(\cdot) - \int_{\R^d\setminus \{0\}} y\chi(y)\, \mu (\cdot, dy)$ is identically zero.
\end{enumerate}
Occasionally we will additionally assume the following \emph{tightness assumption}
\begin{equation}\tag{{\sf Ti}} \label{e.tightness}
\sup_{x\in \R^d} \mu (x, \{|y|>R\}) \to 0 \quad \mbox{as}\quad R\to \infty.
\end{equation}
\end{hyp} 

Let us comment briefly on these assumptions. Assumption (iii) implies that the above operator $\hat B$ is well-defined on $\Cb^\rho(\R^d)$. To see this, let us first assume that $\rho>1$. Then clearly the local part $b(\cdot) \nabla f(\cdot)$ of $\hat B$ is well-defined on this space. As for the integral part,  the elementary estimate
\begin{equation*}
	|f(x+y)-f(x)- y \cdot \nabla f(x) \chi(|y|)| \leq 2\|f\|_{\Cb^{\rho}(\R^d)} \min\{1,|y|^{\rho}\}
\end{equation*} 
implies that $\|\hat{B}f\|_{\infty} \leq 2\|\mu\|_{\beta} \|f\|_{\Cb^{\rho}}$. In the case $\rho\leq 1$, the additional assumption (iv) entails that the operator $\hat B$ simplifies to 
\[
\hat B f(x) = \int_{\R^d\setminus\{0\}} (f(x+y) - f(x))\, \mu (x, dy)
\]
and we can argue similarly, making use of the H\"older continuity of $f$. The assumption $\beta<\rho$ is important to ensure compatibility (in the sense of Hypothesis \ref{h1}) of the perturbation $\hat B$ with the semigroup $T$ satisfying Hypothesis \ref{h1b}. The  tightness condition \eqref{e.tightness} plays an important role in proving that the perturbed semigroup leaves the space $\Co(\R^d)$ invariant. We note that \eqref{e.tightness} holds if the real part of the symbol 
\begin{equation*}
	p(x,\xi) := -ib(x) \cdot \xi + \int_{\R^d \setminus \{0\}} (1-e^{iy \cdot \xi} + iy \cdot \xi \chi(|y|)) \, \mu(x,dy)
\end{equation*}
of the operator $\hat{B}$ is equicontinuous at $\xi=0$, i.e.
\begin{equation}
	\lim_{|\xi| \to 0} \sup_{x \in \R^d} |\Re p(x,\xi)|=0,
\end{equation}
see \cite[(proof of) Thm.~4.4]{rs98}. 

We can now formulate the main result of this section.

\begin{thm} \label{t.perturbeuclidean}
Assume Hypotheses \ref{h1b} and \ref{h2}. Then the following hold true.
\begin{enumerate}[(a)]
\item The operator $\hat A + \hat B$ is the full generator of a $\Cb$-semigroup $S= (S(t))_{t\geq 0}$ that consists of strong Feller operators.
\item If $T$ consists of sub Markovian operators then so does $S$.
\item Assume additionally the tightness assumption \eqref{e.tightness},  that $T$ is a positive $\Co$-semigroup and that the test functions $\Cc^\infty(\R^d)$ form a core for the generator $A$ on $\Co(\R^d)$. Then $S$ is a $\Co$-semigroup. 
\end{enumerate}
\end{thm} 

\begin{rem} \label{r.markovian}
In the situation of Theorem \ref{t.perturbeuclidean}(b), the perturbed semigroup $S$ is Markovian if and only if the unperturbed semigroup $T$ is Markovian. Indeed, we only need to check if the semigroup is conservative, i.e.\ if the semigroup leaves the function $\one$ invariant. By \cite[Prop.\ 5.9]{k11} this is the case if and only if $\one$ belongs to the kernel of the full generator. As $\hat B \one=0$, we see that 
$(\hat A + \hat B)\one = 0$ is equivalent to $\hat A \one =0$.
\end{rem}

\begin{example}\label{ex.c0notinv}
Without the tightness assumption \eqref{e.tightness}, it is not true in general that the perturbed semigroup $S$ is a $\Co$-semigroup even if this is the case for $T$. This can be seen for example by considering in dimension $d=1$ the kernel $\mu (x, \cdot) \coloneqq \delta_{-x}$, which obviously does not satisfy \eqref{e.tightness}.  Choose $b=0$ and let $T$ be the heat semigroup. Then $\hat B f = f(0)\one_\R - f$ and the perturbed semigroup $S$ is (up to a rescaling by the factor $e^{-t}$) the semigroup from Example \ref{ex.noc0}, which is not a $\Co$-semigroup.
\end{example}

\begin{cor}\label{c.core}
Assume Hypotheses \ref{h1b} and \ref{h2} and that $T$ is a positive $\Co$-semigroup. If $\Cc^\infty(\R^d)$ is a core for the generator $A$ of $T|_{\Co(\R^d)}$, then $\{ (f, Af +\hat B f) : f\in \Cc^\infty(\R^d)\}$ is a bp-core for the full generator $\hat A +\hat B$ of $S$.
\end{cor}

\begin{proof} 
If $\Cc^{\infty}(\R^d)$ is a core for the generator on $\Co(\R^d)$, then, by Lemma~\ref{l.bp}, $\{(f,Af); f \in \Cc^{\infty}(\R^d)\}$ is a bp-core for $\hat{A}$. Thus, by Theorem~\ref{t.perturbFeller}(a), $\{(f,Af+\hat{B}f); f \in \Cc^{\infty}(R^d)\}$ is a bp-core for the full generator $\hat{A}+\hat{B}$.
\end{proof}

We now turn to the proof of Theorem~\ref{t.perturbeuclidean}. 

\begin{proof}[Proof of parts (a) and (c) of Theorem \ref{t.perturbeuclidean}]

(a) Let us first prove prove that $\hat B f \in \Bb(\R^d)$ for $f\in \Cb^\rho(\R^d)$. This is obvious for the local part of $\hat B f$ so we focus on the integral part. From Hypothesis \ref{h2}(ii), it follows that for any Borel subset $S$ of $\R^d\setminus\{0\}$ the map
$x\mapsto \int_{\R^d\setminus\{0\}} \one_S(y)\, \mu (x, dy)$ is measurable. But then so is $x\mapsto \int_{\R^d\setminus\{0\}} \one_{R\times S}(x,y)\, \mu (x, dy)$ for Borel sets $R\subset \R^d$ and $S\subset \R^d\setminus\{0\}$. An application of the monotone class theorem yields that $x \mapsto \int_{\R^d \setminus\{0\}} \I_A(x,y) \, \mu(x,dy)$ is Borel measurable for any $A \in \cB(\R^d) \otimes \cB(\R^d \setminus \{0\})$ and thus, by the sombrero lemma and dominated convergence, $x \mapsto \int_{\R^d \setminus \{0\}} g(x,y) \, \mu(x,dy)$ is Borel measurable for any function $g$ which is $\cB(\R^d) \otimes \cB(\R^d \setminus \{0\})$-measurable and satisfies $|g(x,y)| \leq K \min\{1,|y|^{\rho}\}$ for some constant $K>0$. This readily gives that $\hat{B}f$ is Borel measurable. Note that the discussion before Theorem \ref{t.perturbeuclidean} shows that, in fact, $\hat B$ is a bounded linear operator from $\Cb^\rho(\R^d)$ to $\Bb(\R^d)$. \smallskip

We now verify Hypothesis \ref{h1}, then part (a) immediately follows from Theorem \ref{t.k13}. As for condition (i), the boundedness of $\hat B T(t)$ follows from Hypothesis \ref{h1b}(i) and the boundedness of $\hat B$ proved above. As for the $\sigma$-continuity, let a bounded sequence $(f_n)_{n\in \N} \subset \Bb(\R^d)$ be given such that $f_n \to f$ pointwise. By Hypothesis \ref{h1b}(i), $T(t)f_n$ is bounded in
$\Cb^\rho(\R^d)$. By the Arzel\`a--Ascoli theoren, passing to a subsequence, we may and shall assume that $T(t)f_n$ converges locally uniformly to some continuous function. By the $\sigma$-continuity of $T(t)$ the sequence $T(t)f_n$ converges pointwise to $T(t)f$, whence the only possible limit is $T(t)f$ and it follows that $T(t)f_n \to T(t)f$ locally uniformly. Note that in the case where $\rho>1$ we also obtain that $\nabla T(t)f_n \to \nabla T(t)f$ locally uniformly. From this it is immediate that the local part of $\hat Bf_n$ converges pointwise to that of $\hat B f$. As for the integral part this convergence follows from dominated convergence, noting that as a consequence of the uniform boundedness in $\Cb^{\rho}(\R^d)$ we find an integrable majorant of the form
$C\min\{1, |y|^\rho\}$.

The proof of Condition \ref{h1}(ii) is similar, taking Lemma \ref{l.hoelderdomain} into account. As for Condition \ref{h1}(iii), we note that $(t,x) \mapsto T(t) f(x)$ is Borel measurable for every $f \in \Bb(\R^d)$ and using a reasoning similar to that in the first part of this proof, it follows that  $(t,x) \mapsto \hat{B}T(t) f(x)$ is Borel measurable. 
 
Condition \ref{h1}(iv) is an immediate consequence of Hypothesis \ref{h1b}(ii) and the boundedness of $\hat B$.\medskip

(c) To prove this part, we use Theorem \ref{t.perturbFeller}. We thus have to prove that $\hat B T(t) \Bo(\R^d)\subset \Bo(\R^d)$. We will only consider the case $\rho>1$; for $\rho \leq 1$ the reasoning is a bit simpler because all terms involving the gradient vanish by Hypothesis~\ref{h2}(iv).  Take $f\in \Cb^\rho(\R^d)\cap \Cc(\R^d)$ and choose $R>0$ such that the support of $f$ is contained in the ball $B(0,R)$. Taking into account that in this part we assume the tightness condition \eqref{e.tightness}, we find that if $|x|>R+r$ for some $r>0$, then
\begin{align*}
			|\hat{B} f(x)| 
			= \left| \int f(x+y) \, \mu(x,dy) \right|
			&\leq \|f\|_{\infty} \int_{\{|y| \geq r\}} \mu(x,dy) 
			\to 0
\end{align*}
as $r\to \infty$. Thus, $\hat B f \in \Bo(\R^d)$. 

Now let $g\in \Co(\R^d)$ be given. We denote the generator of $T|_{\Co(\R^d)}$ by $A$. As $\Cc^\infty(\R^d)$ is a core, we find a sequence $(f_n)\subset \Cc^\infty(\R^d)$ such that
$g_n \coloneqq \lambda f_n - A f_n  \to g$, see \cite[Ex. II.1.15]{en}. It then follows that $f_n \to f=R(\lambda)g \in D(A)$. 
By the above, $\hat B R(\lambda)g_n = \hat B f_n \in \Bo(\R^d)$.
As $\hat BR(\lambda)$ is $\|\cdot\|_\infty$-continuous, it follows that $\hat B f = \hat B R(\lambda )g = \lim_{n\to \infty} \hat BR(\lambda) g_n = \lim_{n\to\infty}\hat B f_n$ also belongs to $\Bo(\R^d)$. Consequently, $\hat B f \in \Bo(\R^d)$ whenever $f\in D(A)$.

As $T(t)D(A) \subset D(A)$ for every $t>0$, it follows that $\hat B T(t)f \in \Bo(\R^d)$ for $f\in D(A)$. But as $D(A)$ is dense in $\Co(\R^d)$ and $\hat B T(t)$ is $\|\cdot\|$-continuous, this  is also true for $f\in \Co(\R^d)$.
\end{proof}

\begin{rem}\label{rem.nocore}
Theorem \ref{t.perturbeuclidean}(c) remains valid also without the assumption that $\Cc^\infty(\R^d)$ is a core for $T|_{\Co(\R^d)}$, provided we assume that $b\in \Bo(\R^d; \R^d)$. 

To see this,  pick a sequence $\phi_n \in \Cc^\infty(\R^d)$ with $\one_{B(0,n)} \leq \phi_n \leq \one_{B(0,2n)}$ such that  $\sup_n \|\phi_n\|_{\Cb^2(\R^d)}<\infty$. Given $f\in \Cb^\rho(\R^d) \cap \Co(\R^d)$, we put $f_n \coloneqq f\phi_n$. Then $f_n \in \Cb^\rho(\R^d)\cap \Cc(\R^d)$ and $f_n$ is a bounded sequence that converges to $f$ uniformly on $\R^d$; moreover, $\nabla f_n$ is bounded and converges to $\nabla f$ locally uniformly.

As $b\in \Bo(\R^d; \R^d)$ and the sequence $\nabla f_n$ is uniformly bounded, it is easy to see that $b\nabla f_n \to b\nabla f$ with respect to $\|\cdot\|_\infty$. Making use of the boundedness of $(f_n)$ in $\Cb^\rho(\R^d)$ and the tightness assumption \eqref{e.tightness} we can prove that also the nonlocal part converges with respect to $\|\cdot\|_\infty$. . In conclusion, $\hat B f_n \to \hat B f$ with respect to $\|\cdot\|_\infty$. As the proof of Theorem \ref{t.perturbeuclidean}(b) shows that  $\hat B f_n \in \Bo(\R^d)$, the same is true for its uniform limit $\hat B f$. This proves that $\hat B f\in \Bo(\R^d)$ for all $f\in \Cb^\rho(\R^d)\cap\Co(\R^d)$ which is enough to ensure
that $\hat BT(t)$ maps $\Co(\R^d)$ to $\Bo(\R^d)$.
\end{rem}

To prove part (b) of Theorem \ref{t.perturbeuclidean}, we will employ the convergence result of Proposition \ref{p.resolventconvergence}. 
As this form of resolvent convergence is of independent interest, we formulate a separate lemma.

\begin{lem} \label{l.resolventconvergence.intdiff}
Assume Hypothesis \ref{h1b} and let sequences $(b_n)$ and $(\mu_n)$ be given that satisfy the assumptions in Hypothesis \ref{h2} for a common constant $\beta \in (0, \min \{2, \rho\})$. Denote the associated operators (see \eqref{eq-intdiff}) by $\hat B_n$. Moreover, we are given functions $b$ and $\nu$  such that
	\begin{enumerate}
		\item\label{l.resolventconvergence.intdiff.i} $b_n(x) \to b(x)$ and $\mu_n(x,\cdot) \to \mu(x,\cdot)$ vaguely for every $x \in \R^d$,
		\item\label{l.resolventconvergence.intdiff.ii} $\sup_{n \in \N}\left( \|b_n\|_{\infty} + \|\mu_n\|_{\beta}\right)<\infty$ with $\|\cdot\|_{\beta}$ defined in \eqref{eq-kernel.integrability}.
		\item\label{l.resolventconvergence.intdiff.iii} $\sup_{x\in K}\sup_{n\in \N}\mu_n (x, \{ |y|>R\}) \to 0$ as $R\to \infty$ for $K\Subset \R^d$.
	\end{enumerate}
	Then, denoting the operator associated to $b$ and $\mu$ by $\hat B$, we find
	\begin{equation*}
		(\lambda-(\hat{A}+\hat{B}_n))^{-1} f \weak (\lambda-(\hat{A}+\hat{B}))^{-1} f, \qquad f \in \Bb(\R^d),
	\end{equation*}
	for large enough $\lambda$.
\end{lem}

\begin{proof}
	We give the proof only for $\rho>1$. If $\rho \leq 1$, then by Hypothesis \ref{h2}(iv) all terms involving the gradient vanish in the below computations; apart from that the reasoning is analogous. 
	
Without loss of generality, we may assume $\beta \geq 1$ (otherwise consider $\tilde{\beta} := \max\{1,\beta\}$). We first note that also the functions $b$ and $\nu$ satisfy Hypothesis \ref{h2} (see Lemma \ref{l.kalphabpclosed}), whence Theorem \ref{t.perturbeuclidean}(a) yields that $\hat A + \hat B$ and, for every $n\in \N$, the operator $\hat A + \hat B_n$ is the full generator of a $\Cb$-semigroup. The discussion following Hypothesis~\ref{h2} shows that the operator $\hat B_n$ defines a bounded linear operator from $\Cb^{\beta} (\R^d)$ to $\Bb(\R^d)$. Using our assumption (ii) above, we actually see that there is a constant $K>0$ such that
	\begin{equation}
		\|\hat{B}f\|_{\infty}  + \sup_{n \in \N} \|\hat{B}_n f\|_{\infty} \leq K \|f\|_{\Cb^{\beta}(\R^d)}, \qquad f \in \Cb^{\beta}(\R^d).
		\label{eq-bn.bdd}
	\end{equation}
We now check the assumptions of Proposition \ref{p.resolventconvergence}.\medskip

\emph{Assumption (i):} By Lemma \ref{l.hoelderdomain}, we have
\[
\|\hat B_n R(\lambda)\| \leq \|\hat B_n\|_{\cL (\Cb^\beta, \Bb)}\|R(\lambda)\|_{\cL(\Bb, \Cb^\beta)}\leq KC(\lambda) \to 0
\]
as $\lambda \to \infty$.\smallskip

	\emph{Assumption (ii):} In view of Lemma \ref{l.hoelderdomain}, it suffices to prove that $\hat{B}_n f \weak \hat{B}f$ for every $f \in \Cb^{\rho}(\R^d)$, so fix $f \in \Cb^{\rho}(\R^d)$. Because of \eqref{eq-bn.bdd}, it actually suffices to show pointwise convergence.  Since $b_n$ converges pointwise to $b$, it is clear that the local part of $\hat{B}_nf$ converges to the local part of $\hat{B}f$; therefore we assume in the following that the local part is zero. Write
	\begin{equation*}
		\hat{B}_n f = I_n^{(r)} f + J_n^{(r,R)}f + K_n^{(R)} f,
	\end{equation*}
	where 
	\begin{align*}
		I_n^{(r)} f(x) &:= \int_{\{|y| \leq r\}} (f(x+y)-f(x)-\nabla f(x) \cdot y) \, \mu_n(x,dy) \\
		J_n^{(r,R)} f(x) &:= \int_{\{r<|y| <R\}} (f(x+y)-f(x)-\nabla f(x) \cdot y \chi(y)) \, \mu_n(x,dy)  \\
		K_n^{(R)} f(x) &:= \int_{\{|y| \geq R\}} (f(x+y)-f(x)) \, \mu_n(x,dy) 
	\end{align*}
	for $r \in (0,1)$ and $R \in (1,\infty)$. We decompose the operator $\hat{B}$ in the same way, and we study the convergence for each of the terms separately. By another application of Taylor's formula,
	\begin{equation*}
		\sup_{n \in \N} |I_n^{(r)} f(x)| 
		\leq  \|f\|_{\Cb^{\rho}(\R^d)}  \sup_{n \in \N} \int_{\{|y| \leq r\}} |y|^{\rho} \, \mu_n(x,dy)
	\end{equation*}
	As $\rho>\beta$, it follows from \ref{l.resolventconvergence.intdiff.ii} that the right-hand side is finite and converges to $0$ (uniformly in $x$) as $r \to 0$. As $\|\mu\|_{\beta}<\infty$, an analogous estimate holds for $I^{(r)} f$. Using 
	\begin{equation*}
		|K_n^{(R)} f(x)| \leq 2 \|f\|_{\infty} \int_{\{|y| \geq R\}} \mu_n(x,dy)
	\end{equation*}
	and the corresponding estimate for $K^{(R)} f$, assumption (iii) implies that  $|K_n^{(R)} f(x)| + |K^{(R)} f(x)| \to 0$ uniformly in $n \in \N$ and $x$ in compact subsets of $\R^d$  as $R \to \infty$. As $f$ and $\nabla f$ are continuous, the vague convergence $\mu_n(x) \to \mu(x)$ entails that $J_{n}^{r,R}f(x) \to J^{r,R} f(x)$ for any $r,R>0$ with  $\mu(x,\{|y|=r\})=0$ and $\nu(x,\{|y|=R\})=0$. Since there are for each fixed $x$ at most countably many radii with $\mu(x,\{|y|=r\})>0$,  we can let $r \to 0$ and $R \to \infty$ along suitable sequences to deduce that $\hat{B}_n f(x) \to \hat{B}f(x)$. \smallskip
	
\emph{Assumption (iii)}: Making use of Lemma \ref{l.hoelderdomain} again, we see that it suffices to prove that
	$\hat{B}_n u_n \weak 0$ for any sequence $(u_n)_{n \in \N}$ with $u_n \weak 0$ and $\sup_{n\in \N} \|u_n\|_{\Cb^{\rho}}<\infty$.
	
	Fix such a sequence.  Then $\sup_{n \in \N} \|\hat{B}_nu_n\|_{\infty}<\infty$ is immediate from \eqref{eq-bn.bdd} and it only remains to prove pointwise convergence. Pick $\rho' \in (\beta, \rho)$. By Taylor's formula, 
	\begin{align*}
		|\hat{B}_n u_n(x)| 
		&\leq \|b_n\|_{\infty} |\nabla u_n(x)| + \|u_n\|_{\Cb^{\rho'}(B[x,R])} \int_{\{|y| \leq R\}} |y|^{\rho'} \, \mu_n(x,dy)  \\
		&\quad + 2 \|u_n\|_{\infty} \int_{\{|y|>R\}} \, \mu_n(x,dy)
	\end{align*}
	for any $R>1$; here $B[x,r]$ is the closed ball around $x$ with radius $R$. Using \ref{l.resolventconvergence.intdiff.ii}, the fact that $\rho>1$ and that  $(b_n)_{n \in \N}$ is bounded, we find that given $\varepsilon>0$ there is some $R \gg 1$ such that 
	\begin{equation*}
		|\hat{B}_n u_n(x)|
		\leq C \|u_n\|_{\Cb^{\rho'}(B[x,R])} + \varepsilon 
	\end{equation*}
	for a finite constant $C=C(R)$. In order to deduce that the left-hand side converges to $0$, we use the subsequence principle. As $(u_n)$ is bounded in $\Cb^\rho$, by compact embedding, there is a subsequence $(u_{n''})_{n'' \in \N}$ which  converges on the closed ball $B[x,r]$ with respect to the $C^{\rho'}$-norm. Since we know that $u_n \to 0$  pointwise, the $C^{\rho}$-limit is also zero and so
	\begin{equation*}
		\limsup_{n'' \to \infty} |\hat{B}_{n''} u_{n''}(x)| \leq \varepsilon.
	\end{equation*}
	Hence, $\limsup_{n \to \infty} |\hat{B}_n u_n(x)| \leq \varepsilon$. As $\varepsilon>0$ is arbitrary, we conclude that $\hat{B}_n u_n(x) \to 0$ for all $x \in \R^d$.\medskip
	
Now Proposition~\ref{p.resolventconvergence} yields the claim.
\end{proof}

We can now finish the proof of Theorem \ref{t.perturbeuclidean}

\begin{proof}[Proof of Theorem~\ref{t.perturbeuclidean}(b)]

Let us first only consider kernels $\mu$ satisfying
\begin{equation}\label{eq.localtight}
\sup_{x\in K} \mu (x, \{|y|> R\}) \to 0\quad\mbox{as}\quad R\to \infty\qquad \mbox{ for all } K\Subset \R^d.
\end{equation}

Given $b$ and $\mu$ as in Hypothesis \ref{h2}, we write $\hat B_{b,\mu}$ for the operator defined via \eqref{eq-intdiff} with these particular coefficients and $S_{b,\mu} = (S_{b,\mu}(t))_{t\geq 0}$ for the semigroup generated by the operator $\hat A + \hat B_{b,\mu}$, which exists by part (a) of Theorem \ref{t.perturbeuclidean}. Denote by $F$ the family of pairs $(b,\mu)$ for which the assertion holds, i.e.\
\begin{equation*}
	F= \{(b,\mu) \, \, \text{satisfying \eqref{eq.localtight} and Hypothesis~\ref{h2} }\::\: \text{$S_{b,\mu}$ is sub Markovian}\}.
\end{equation*}
It follows from Lemma \ref{l.resolventconvergence.intdiff} and Corollary \ref{c.stability} that $F$ is bp-closed. Therefore, it suffices to show that $(b,\mu) \in F$ for any $b \in \Cc(\R^d;\R^d)$ and any $\mu \in \Cc\big(\R^d, \mathcal{M}^+(\R^d\setminus \{0\})\big)$ with $\|\mu\|_{\beta}<\infty$ and \eqref{eq.localtight}. This is a consequence of the known fact that $\Cc(\R^d;\R^d)$ is bp-dense in $\Bb(\R^d;\R^d)$ and an analogous result for measures which we establish in Appendix~\ref{appa}; see in particular Corollary~\ref{c.bpdense} (with $g(y)\coloneqq \min\{|y|^{\beta},1\}$).\smallskip

So fix $b \in \Cc(\R^d;\R^d)$ and $\mu \in \Cc\big(\R^d, \mathcal{M}^+(\R^d\setminus \{0\})\big)$ with \eqref{eq.localtight}. Then, in particular, Theorem \ref{t.perturbeuclidean}(b) and Remark~\ref{rem.nocore} show that $S_{b, \mu}$ is a $\Co$-semigroup; let us denote the generator of the restriction of $S_{b,\mu}$ to $\Co(\R^d)$ by $L_{b, \mu}$. 

It was seen in the proof of Theorem \ref{t.perturbeuclidean}(b) that
for $f\in \Cb^\rho(\R^d)\cap \Co(\R^d)$ we have $\hat B_{b,\nu} f\in \Bo(\R^d)$. 
Moreover, $\|\mu\|_{\beta} < \infty$ and \eqref{eq.localtight} give
	\begin{equation}
			\limsup_{R \to \infty} \sup_{x \in K} \int_{\{|y|>R\}} \, \mu(x,dy)=0 
			\qquad \lim_{r \to 0} \sup_{x \in K} \int_{\{|y| \leq r\}} |y|^2 \, \mu(x,dy)=0,
			\label{eq-tightness}
	\end{equation}
	for all $K\Subset \R^d$. This entails continuity of $\hat{B}_{b,\nu}f$, see Theorem~\ref{appb-1} (note that we can assume without loss of generality that $\chi$ is smooth by modifying the drift term accordingly). Hence, $\hat{B}_{b, \mu} f\in \Co(\R^d)$.\smallskip
	
Denoting by $A$ the generator of the restriction of $T$ to $\Co(\R^d)$, we may consider $L\coloneqq A+\hat B_{b, \mu}$, defined on $D(A)$, as an operator on the space $\Co(\R^d)$. We note that  $\hat B_{b, \mu} R(\lambda)\Co(\R^d) \subset \Co(\R^d)$. and from this and  Lemma \ref{l.perturbedresolvent} it follows that, for large enough $\lambda$, the operator $\lambda - L$ is surjective.

As is well known (see e.g.\ \cite[Thm.\ 4.2.2]{ek}) a strongly continuous semigroup on $\Co(\R^d)$ is sub Markovian if and only if its generator satisfies the positive maximum principle. By assumption, this is certainly true for $A$. However, if $f\in D(A)$ satisfies $f(x_0) = \max\{ f(x) : x\in \R^d\} \geq 0$, then $f(x_0 + y) - f(y) \geq 0$ for all $y\in \R^d$; if $\rho>1$, we also see that $\nabla f(x_0)= 0$. It follows from $D(A) \subseteq \Cb^{\rho}(\R^d)$, cf.\ Lemma~\ref{l.hoelderdomain}, and the definition of the operator $\hat B_{b, \mu}$ and that $\hat B_{b, \mu} f(x_0) \geq 0$. This shows that
$L= A+\hat B_{b, \mu}$ satisfies the positive maximum principle.

By \cite[Thm.\ 4.2.2]{ek}, $L$ generates a sub Markovian semigroup on $\Co(\R^d)$. However, $L_{b, \nu}$ is obviously an extension of $L$, so we must have $L=L_{b,\nu}$ and thus the restriction of $S$ to $\Co(\R^d)$ is sub Markovian and so is then also $S$ itself. This shows that
$(b,\nu) \in F$ and finishes the proof in this case.\medskip

To remove the tightness assumption \eqref{eq.localtight}, we consider the `small' and the `large' jumps created by $\mu$ separately. To that end, we put
\begin{equation}
\mu_s (x, dy) \coloneqq \one_{\{|y|\leq 1\}}\mu (x, dy) \qquad\mbox{and}\qquad \mu_l(x, dy) \coloneqq \one_{\{|y|> 1\}} \mu (x, dy).
\label{e.splitOperator}
\end{equation}
We then split the operator $\hat B = \hat B_1 + \hat B_2$ where, in the notation above, $\hat B_1 = B_{b, \mu_s}$ and $\hat B_2 = \hat B_{0, \mu_l}$. As $\mu_s(x, \{|y|>1\}) \equiv 0$, we can apply the above to infer that $\hat A + \hat B_1$ is the full generator of a sub Markovian semigroup. Noting that $\hat B_2$ is a bounded operator and taking the special structure of this operator into account, the results of \cite[Sect.\ 4.10]{ek} (see in particular Exercise 3 on p.\ 261) yield that also $\hat A + \hat B = \hat A + \hat B_1 +\hat B_2$ is the full generator of a sub Markovian semigroup.
\end{proof}

\section{Application: Martingale problems} \label{mp}

We apply the results from the previous section to study perturbations of martingale problems associated with L\'evy-type operators. We do not strive for full generality; we rather would like to illustrate how our results can be applied in this setting. We will make the following assumptions
which ensure that $T$ and the perturbed semigroup $S$ are Feller semigroups, provided the tightness assumption \eqref{e.tightness} is satisfied. Note, however, that \eqref{e.tightness} is not assumed in this section.

\begin{hyp} \label{hyp.mp}
	Assume that $T=(T(t))_{t \geq 0}$ is a semigroup of kernel operators such that
	\begin{enumerate}
		\item $(T(t))_{t \geq 0}$ satisfies Hypothesis~\ref{h1b} for some constant $\rho>0$,
		\item $(T(t))_{t \geq 0}$ is a Markovian $\Co$-semigroup,
		\item the test functions $\Cc^{\infty}(\R^d)$ are a core for the generator $A$ of $T|_{\Co(\R^d)}$. 
	\end{enumerate}
\end{hyp}

It follows from general theory (see, e.g., \cite[Sect. 3.2]{revuz-yor}) that we can associate Markov process (actually a Feller process) with c\`adl\`ag paths with the semigroup $T$, i.e.\ there is a Markov process whose transition semigroup equals $T$. 

There is another connection between stochastic processes and (Feller) semigroups via \emph{martingale problems}. Let us briefly recall the relevant notions. By $\cD([0,\infty); \R^d)$ we denote the \emph{Skorohod space}, i.e.\ the space of c\`adl\`ag functions $\omega: [0,\infty) \to \R^d$; see \cite{bill} for more information. Given a  (possibly multi-valued) operator $L$ with domain $D(L)$ , a set $D\subset D(L)$ and a measure $\mu$ on $\R^d$, a measure $\PP$ on $\cD([0,\infty); \R^d)$ is called \emph{a solution to the $(L, D; \mu)$-martingale problem} if (i) $\PP (\omega_0 \in A) = \mu (A)$ for all Borel sets $A\subset \R^d$ and (ii) the process 
\[
M_t^{(f,g)}(\omega) \coloneqq f(\omega_t) - f(\omega_0) - \int_0^t g(\omega_s)\, ds, \quad w \in \cD[0,\infty),
\]
is a martingale under $\PP$ with respect to the canonical filtration $(\sigma( \omega_s : s\leq t))_{t\geq 0}$ for any $(f,g) \in L$ with $f\in D$.
In case that $L$ is single valued on $D$, we have $g=Lf$. 

We say that \emph{uniqueness in law} holds for the $(L,D)$-martingale problem if any two solutions $\PP_1$, $\PP_2$ with identical initial distribution satisfy $\PP_1=\PP_2$. The $(L,D)$-martingale problem is \emph{well-posed} if for every $x \in \R^d$ there exists a unique solution $\PP^x$ to the $(L,D,\delta_x)$-martingale problem.\smallskip

If $(X_t)$ is a Markov Process associated with the semigroup $T$, then its distribution solves  the $(\hat A, D(\hat A))$-martingale problem for the full generator $\hat A$ of $T$, see \cite[Prop. 4.1.7]{ek}. 
Conversely, if the $(L, D)$-martingale problem is well-posed, then under each measure $\PP^x$, the canonical process $(\omega_t)_{t\geq 0}$ is a Markov process \cite[Thm.\ 4.4.2]{ek}; if $(L,D)$ is, in a sense, `rich enough' to determine a semigroup uniquely, this semigroup is the transition semigroup of the process, see \cite[Thm.\ 4.4.2]{ek}.

\begin{thm} \label{t.mp.wellposed}
	Let $(T(t))_{t \geq 0}$ be a semigroup satisfying Hypothesis~\ref{hyp.mp} with full generator $\hat{A}$ and denote by $A$ the generator of $T|_{\Co(\R^d)}$. If $\hat{B}$ is a L\'evy-type operator
	\begin{equation*}
		\hat{B}f(x) = b(x) \nabla f(x) + \int_{\R^d \setminus \{0\}} (f(x+y)-f(x)-\nabla f(x) \cdot y \chi(y)) \, \mu(x,dy)
	\end{equation*}
	with $b(x)$ and $\nu(x,dy)$ satisfying Hypothesis~\ref{h2}, then the $(A+\hat B, \Cc^\infty(\R^d))$-martingale problem is well-posed.
\end{thm}

\begin{proof}
By Theorem~\ref{t.perturbeuclidean}, there is a  Markovian $\Cb$-semigroup $S=(S(t))_{t \geq 0}$  with full generator $\hat{A}+\hat{B}$. By Corollary \ref{c.core}, $\{ (f, Af +\hat Bf) : f\in \Cc^\infty(\R^d)\}$ is a bp-core for $\hat A + \hat B$. By \cite[Prop.\ 4.3.1]{ek},
a measure $\PP$ on $\cD([0,\infty);\R^d)$ solves the $(A+\hat B, \Cc^\infty(\R^d))$-martingale problem if and only if it solves the
$\hat A + \hat B$-martingale problem. It follows from \cite[Thm.\ 4.4.1]{ek}, applied with $L= \overline{D(\hat A)}$, which certainly includes
$\Co(\R^d)$ and is thus separating, that uniqueness holds for the $\hat A + \hat B$-martingale problem (and thus for the $(A+\hat B, \Cc^\infty(\R^d))$-martingale problem).

It thus remains to establish existence of solutions. To that end, we first additionally assume that the tightness assumption \eqref{e.tightness} is satisfied. In this case, Theorem \ref{t.perturbeuclidean} yields that $S$ is actually a Feller semigroup and as is well known, see, e.g.\ \cite[Theorem 4.10.3]{kol}, this implies well-posedness for the martingale problem for the generator of $S|_{\Co}$. In particular, there is a solution to the $(A+\hat B, \Cc^\infty(\R^d))$-martingale problem. To remove the tightness condition, we proceed as in the proof of Theorem~\ref{t.perturbeuclidean}(b), i.e.\ we write $\hat{B}=\hat{B}_1+\hat{B}_2$ for a bounded operator $\hat{B}_2$ and a L\'evy-type operator $\hat{B}_1$, whose jumping kernels are supported in the unit ball, cf.\ \eqref{e.splitOperator}. The existence of a solution to the $(A+\hat B_1, \Cc^\infty(\R^d))$-martingale problem is clear from the first part. Since $\hat{B}_2$ is a bounded perturbation, existence of a solution to the martingale problem for $\hat{A}+\hat{B}=\hat{A}+\hat{B}_1+\hat{B}_2$ follows from \cite[Prop.\ 4.10.2]{ek}.
\end{proof}

We can now also prove continuous dependence of the solutions of our martingale problem on the `coefficients' $b$ and $\mu$. To that end, we make the same assumptions as in Lemma \ref{l.resolventconvergence.intdiff}. As a direct consequence, we obtain a convergence result for semigroups, which fits well to our earlier result on the resolvent convergence, cf.\ Proposition~\ref{p.resolventconvergence} and Corollary~\ref{c.strongerconvergence}.

\begin{thm} \label{t.sg.conv}
	 Let $(T(t))_{t \geq 0}$ be a semigroup satisfying Hypothesis~\ref{hyp.mp} for some constant $\rho>0$ and denote by $\hat{A}$ its full generator. Let $(b_n)_{n \in \N\cup\{\infty\} }$ and $(\mu_n)_{n \in \N\cup\{\infty\}}$ satisfy Hypothesis \ref{h2} for a common constant $\beta \in (0, \min \{2, \rho\})$. Denote the associated L\'evy-type operators by $\hat B_n$ for $n\in\N\cup\{\infty\}$. Assume that
		\begin{enumerate}
			\item $b_n(x) \to b_\infty(x)$ and $\mu_n(x,\cdot) \to \mu_\infty (x,\cdot)$ vaguely for every $x \in \R^d$,
			\item $\sup_{n \in \N}\left( \|b_n\|_{\infty} + \|\mu_n\|_{\beta}\right)<\infty$ with $\|\cdot\|_{\beta}$ defined in \eqref{eq-kernel.integrability}.
			\item $\sup_{x\in K}\sup_{n\in \N}\mu_n (x, \{ |y|>R\}) \to 0$ as $R\to \infty$ for $K\Subset \R^d$.
	\end{enumerate}
Then, if $\PP_n$ solves the $(A+\hat B_n, \Cc^\infty(\R^d), \delta_{x_n})$-martingale problem and $\PP_\infty$ solves the 
	$(A+\hat B_\infty, \Cc^\infty(\R^d), \delta_{x})$-martingale problem for some sequence $x_n \to x$, we have weak convergence $\PP_n \rightharpoonup \PP_\infty$.
	
	In particular, if $S_n(t)$ and $S_\infty(t)$ denote the perturbed semigroups with generator $\hat{A}+\hat{B}_n$ and $\hat{A}+\hat{B}_\infty$, respectively, then $S_n(t)f \to S_\infty(t)f$ pointwise for every $f \in \Cb(\R^d)$.
\end{thm}

\begin{proof}
It was seen in the proof of Lemma~\ref{l.resolventconvergence.intdiff} that under the assumptions above 
$\hat B_n u \weak \hat B_\infty u$ for all $u\in D(\hat A)$. Thus, if $(u, f) \in \hat A$, we have $(u, f + \hat B_n u ) \in \hat A + \hat B_n$ and this bp-converges to $(u, f + \hat B_\infty u)$. In particular, the sequence $g_n \coloneqq f+\hat B_n u$ is uniformly bounded. By assumption,
\[
M^{(u,g_n)}_t(\omega) \coloneqq u(\omega_t) - u(\omega_0) + \int_0^t g_n(\omega_s)\, ds
\]
is a $\PP_n$-martingale. Combining \cite[Thm.\ 4.9.4]{ek} and \cite[Cor.\ 3.9.3]{ek}, it follows that the sequence $(\PP_n)$ is tight. For an alternative argument yielding the tightness, see \cite[Corollary 3.9]{kuehn-mp}.\smallskip

 We will invoke the subsequence principle, see e.g.\ \cite[Thm.\ 2.6]{bill}, to prove that $\PP_n \to \PP$ weakly, where $\PP$ is the distribution of $X$.\smallskip

To that end, observe that as a further consequence of Lemma \ref{l.resolventconvergence.intdiff}, we have
	\begin{equation*}
		R_n(\lambda) f:= (\lambda-(\hat{A}+\hat{B}_n))^{-1} f \weak R_{\infty}(\lambda)f:= (\lambda-(\hat{A}+\hat{B}))^{-1} f, \quad f \in \Bb(\R^d),
	\end{equation*}
	for large enough $\lambda$. In fact, Corollary~\ref{c.strongerconvergence} implies that the convergence is actually uniform on compact sets.
 
Now, take any subsequence of $(\PP_n)_{n \in \N}$, then, by tightness, there is a further subsequence converging to some measure, say, $\mathds{Q}$. For simplicity of notation, we denote the convergent subsequence also by $\PP_n$. For fixed $f \in \Bb(\R^d)$, $s \leq t$, $0\leq r_1 < \ldots < r_k \leq s$ and $h_j \in \Cb(\R^d)$, define
	\begin{equation*}
		\Phi_n(\omega) := \left[ R_n(\lambda) f(\omega_t) - R_n(\lambda) f(\omega_s) - \int_s^t (\lambda R_n(\lambda)f-f)(\omega_r) \, dr \right] \prod_{j=1}^k h_j(\omega_{r_j}),
	\end{equation*}
	where $n \in \N \cup \{\infty\}$. Since $R_n(\lambda)f \in D(\hat{A}+\hat{B}_n)$ and $\PP_n$ solves the martingale problem for $\hat{A}+\hat{B}_n$, we find that $\E_{\PP_n}(\Phi_n):=\int \Phi_n \, d\PP_n=0$ for all $n \in \N$. We claim that $\E_{\mathds{Q}}(\Phi_{\infty})=0$. Because of the weak convergence, it holds that $\E_{\PP_n}(\Phi_{\infty}) \to \E_{\mathds{Q}}(\Phi_{\infty})$. If we can show that
	\begin{equation}
		\lim_{n \to \infty} |\E_{\PP_n}(\Phi_n-\Phi_{\infty})| = 0,
		\label{t.sg.conv-e1}
	\end{equation}
	then it follows that $\E_{\mathds{Q}}(\Phi_{\infty})=\lim_n \E_{\PP_n}(\Phi_n)=0$. To prove \eqref{t.sg.conv-e1}, we note that, by the 
	 uniform boundedness of the resolvents, $M:=\sup_{n \in \N} \|\Phi_n\|_{\infty}+\|\Phi_{\infty}\|_{\infty}<\infty$. By tightness, given $\eps>0$ we find $K \Subset \cD([0,\infty); \R^d)$ such that $\PP_n(K^c)\leq \eps$ for all $n\in \N$. Note that $R:= \sup_{y \in K} \sup_{s \leq t} |\omega_s|<\infty$. So, by the locally uniform convergence of the resolvents, we find $N \in \N$ such that
	\begin{equation*}
		|R_n(\lambda) f(\omega_r)-R_{\infty}(\lambda) f(\omega_r)| \leq \eps
	\end{equation*}
	for all $n \geq N$, $r \in [0,t]$ and $y \in K$. Thus, For some constant $C>0$, we have
	\begin{align*}
		|\E_{\PP_n}(\Phi_n-\Phi_{\infty})|
		\leq C\eps + M \eps 
	\end{align*}
for all $n\geq N$. Hence, \begin{equation*}
		\limsup_{n \to \infty} |\E_{\PP_n}(\Phi_n-\Phi_{\infty})|
		\leq (M +C)\eps.
	\end{equation*}
	As $\eps >0$ was arbitrary, this finishes the proof of \eqref{t.sg.conv-e1} and shows that $\E_{\mathds{Q}}(\Phi_{\infty})=0$. This implies that 
	\begin{equation*}
		M_t^{(f)}(\omega) := R_{\infty} (\lambda)f(\omega_t) - \int_0^t (\lambda R_{\infty}(\lambda)f-f)(\omega_r) \, dr, \quad \omega \in \cD([0,\infty);\R^d),
	\end{equation*}
	is a martingale with respect to $\mathds{Q}$ for any $f \in \Bb(\R^d)$. Noting that
	\begin{equation*}
		\{(R_{\infty}(\lambda) f, \lambda R_{\infty}(\lambda)f-f); f \in \Bb(\R^d)\} = \hat{A}+\hat{B},
	\end{equation*}
	and $\mathds{Q}(\omega_0 \in A) = \delta_x(A)$
	this means that $\mathds{Q}$ is a solution to the $(\hat{A}+\hat{B},D(\hat{A}),\delta_x)$-martingale problem and thus, by well-posedness, $\mathds{Q}=\PP_\infty$. Hence, by the subsequence principle, $\PP_n \to \PP_\infty$ weakly. 
\end{proof}

\section{Examples} \label{s.examples}

Our perturbation results for Feller semigroups and martingale problems require two sets of assumptions: one on the original semigroup $T$ (cf.\ Hypothesis~\ref{h1b} resp.\ \ref{hyp.mp}) and one on the perturbation $\hat{B}$ (cf.\ Hypothesis~\ref{h2}). The latter ensures that $\hat{B}$ is indeed a (lower order) perturbation of the generator of $(T(t))_{t \geq 0}$ and involves the parameter $\rho \in (0,2)$ from Hypothesis \ref{h1b}, which characterizes the regularizing properties of $T$. Given this parameter $\rho$, the conditions on $\hat{B}$ are typically easy to check, and so the main work is to verify the assumptions on $T$; in particular, the regularity estimate
\begin{equation*}
	\|T(t) f\|_{\Cb^{\rho}} \leq \varphi(t) \|f\|_{\infty}, \qquad t \in (0,1),\; f \in \Bb(\R^d),
\end{equation*}
for some function $\varphi \in L^1(0,1)$. For a brief summary of some of our main results we refer to Theorem~\ref{t.generic}. In this section, we present examples of semigroups satisfying our assumptions and give some applications, e.g.\ in the theory of stochastic differential equations.

\subsection{Differential operators} \label{s.diff}

Let $\cA$ be a second order differential operators on $\R^d$ of the form
\begin{equation}
\cA f(x) \coloneqq \mathrm{tr}(Q(x)\nabla^2 f(x))
\label{e.2ndOrderOp}
\end{equation}
for $Q(x)=(q_{ij}(x)) \in \R^{d\times d}$. Our next result gives conditions which ensure that the associated semigroup satisfies the assumptions of our main results.

\begin{prop} \label{p.diff}
	Assume that the coefficients $q_{ij}$ are bounded, H\"older continuous, symmetric (i.e.\ $q_{ij} = q_{ji}$) and strictly elliptic in the sense that there exists a constant $\eta>0$ such that
	\[
	\sum_{i,j=1}^d q_{ij}(x)\xi_i\xi_j \geq \eta|\xi|^2, \quad x,\xi \in \R^d.
	\]
	Then the semigroup $T$ associated with \eqref{e.2ndOrderOp} satisfies Hypothesis~\ref{h1b} and \ref{hyp.mp} with arbitrary $\rho\in (0,2)$, and so Theorem~\ref{t.generic} is applicable (with any $\rho \in (0,2)$).
\end{prop}

\begin{proof} 
	First of all, we note that there is a Markovian Feller semigroup $T$ associated with \eqref{e.2ndOrderOp}; this follows from \cite[Theorem 5.11]{dynkin1}, see also \cite[Theorem 8.1.6]{ek}. Moreover, the second step of the proof of \cite[Theorem 5.11]{dynkin1}, with (0.40) replaced by (0.41) and (0.42) from \cite{dynkin2}, actually shows that 
	\begin{equation}
		\|T(t) f\|_{\Cb^2(\R^d)} \leq c t^{-1} \|f\|_{\infty}, \qquad f \in \Bb(\R^d). \label{e.smoothingDiff}
	\end{equation}
	Since we also have $\|T(t) f\|_{\infty} \leq \|f\|_{\infty}$ and the H\"older space $\Cb^{\gamma}(\R^d)$, $\gamma \in (0,2) \setminus \{1\}$, is a real interpolation space between $\Cb(\R^d)$ and $\Cb^2(\R^d)$, see \cite[Thm.\ 2.7.2.1]{triebel78}, an application of a classical interpolation theorem \cite[Sec.\ 1.3.3]{triebel78} yields that $\|T(t) f\|_{\Cb^{\gamma}(\R^d)} \leq M t^{-\rho/2} \|f\|_{\infty}$ for some constant $M=M(\rho)>0$ and any $f \in L^{\infty}$, $\rho \in [0,2] \setminus \{1\}$. Consequently, Hypothesis~\ref{h1b} is satisfied for any $\rho \in (0,2)$. Moreover, $\Cc^{\infty}(\R^d)$ is a core for the generator of $T|_{\Co}(\R^d)$, cf.\ \cite[Theorem 8.1.6]{ek}, and so Hypothesis~\ref{hyp.mp} holds.
\end{proof}

\begin{rem}
	 Under the milder regularity assumption that the diffusion coefficients $q_{ij}$ are uniformly continuous, one can show with some more effort that the Hypothesis~\ref{h1b} is satisfied (for any $\rho \in (0,2)$), e.g.\ using results from \cite{lunardi}. It is, however, not clear whether $\Cc^{\infty}(\R^d)$ is still a core. In dimension $d=1$, one can use that cores are preserved under random time changes, cf.\ \cite[Theorem 4.1]{FK17}, to deduce that $\Cc^{\infty}(\R)$ is indeed a core if the diffusion coefficient is (uniformly) continuous. Thus, Proposition~\ref{p.diff} and Corollary~\ref{c.diffop} below, hold in dimension $d=1$ also for uniformly continuous diffusion coefficients.
\end{rem} 

Combining Proposition~\ref{p.diff} with Theorem~\ref{t.generic}, we get the following corollary.

\begin{cor}\label{c.diffop} 
Let $\mathcal{L}$ be a L\'evy-type operator of the form 
\begin{align} \label{eq.intdiff} \begin{aligned}
\mathcal{L}f(x) &=  b(x)\cdot\nabla f (x) + \mathrm{tr}(Q(x) \nabla^2f(x)) 
\\ &\quad +
\int_{\R^d\setminus\{0\}} (f(x+y) - f(x)-y \cdot \nabla f(x) \I_{(0,1)}(|y|)) \, \mu (x, dy).
\end{aligned} \end{align}
If the drift $b$ is bounded and measurable, the diffusion coefficients are bounded, H\"older continuous, symmetric and strictly elliptic, the jumping kernels depend measurably on $x$ and there is some $\beta \in (0,2)$ such that
\begin{equation*}
	\sup_{x \in \R^d} \left( \int_{\R^d \setminus \{0\}} \min\{|y|^{\beta},1\} \, \mu(x,dy) \right)<\infty,
\end{equation*}
then there is a Markovian $\Cb$-Feller semigroup $T$ whose full generator $\hat{L}$ is the bp-closure of $\{(f,\mathcal{L}f); f \in \Cc^{\infty}(\R^d)\}$. Moreover, the $(\mathcal{L},\Cc^{\infty}(\R^d))$-martingale problem is well-posed. If additionally the family $(\mu(x,dy))_{x \in \R^d}$ is tight, then $T$ is a Feller semigroup.
\end{cor}

As already mentioned in the introduction, the fact that second order (or even higher order) differential operators with lower order coefficients that are merely measurable generate strongly continuous semigroups on $\Co(\R^d)$ is known for a long time (see \cite{stewart}). Here, we also allow lower order \emph{integral terms} that are merely measurable, providing an extension of Taira's results on domains.  Also for the well-posedness of the martingale problem we require less regularity than e.g.\ in the classical results by Stroock \cite{stroock75}, where continuous dependence of $\mu (x, \cdot)$ on $x$ is assumed.

\subsection{L\'evy generators} \label{s.levy}

Next we consider perturbations of operators $\cA$ which are the generators of L\'evy processes,  that is, stochastic processes with c\`adl\`ag sample paths and independent and stationary increments. Recall that, by the L\'evy--Khintchine formula, any L\'evy process can be uniquely characterized (in distribution) by its characteristic exponent $\psi$ or by its L\'evy triplet $(b,Q,\nu)$, cf.\ \cite{sato} for more information. An important example is the fractional Laplacian $-(-\Delta)^{\alpha/2}$, which is the generator of the isotropic $\alpha$-stable L\'evy process. Its characteristic exponent is $\psi(\xi)=|\xi|^{\alpha}$ and the corresponding L\'evy triplet is $(0,0,\nu)$ with
\begin{equation}
\label{eq.alphastable}
\nu (U) = c_{d,\alpha} \int_{U}\frac{1}{|y|^{d+\alpha}}\, dy
\end{equation}
for a normalizing constant $c_{d, \alpha}$. 
\smallskip

In the following, we will assume that the \emph{Hartman--Wintner condition} is satisfied:
	\begin{equation}
		\lim_{|\xi| \to \infty} \frac{\Re \psi(\xi)}{\log |\xi|} = \infty. \label{hw} 
	\end{equation} 
This is a mild growth condition on $\Re \psi$ which ensures the existence of a smooth transition density $p_t$ for each $t>0$ for the associated process, cf.\ \cite{knop13} for a thorough discussion.

\begin{prop}\label{p.levy}
Let $(X_t)_{t \geq 0}$ be a L\'evy process with transition semigroup $(T(t))_{t \geq 0}$ and characteristic exponent $\psi$ satisfying the Hartman--Wintner condition
\eqref{hw}. If there is a constant $\kappa>0$ such that the transition density $p_t$ of $X_t$ satisfies
\begin{equation}\label{eq.gradientestimate}
\int_0^1\Big(\int_{\R^d} |\nabla p_t (x)|\, dx \Big)^\kappa dt < \infty.
\end{equation}
then Hypothesis~\ref{h1b} and \ref{hyp.mp} hold for any $\rho\leq\kappa$, and so Theorem~\ref{t.generic} is applicable (with $\rho\leq\kappa$).
\end{prop}

\begin{proof}
It is well known that $T(t) f(x) = \int f(x+y) p_t(y) \, dy$ is a $\Co$-semigroup and because of the existence of the density, each $T(t)$ is a strong Feller operator, see e.g.\ \cite[pp.~438--39]{J01}. Moreover, $\Cc^{\infty}(\R^d)$ is a core for the generator of $T(t)|_{\Co}$ by \cite[Cor.~2.10]{ltp3}. Denote by $\hat{A}$ the corresponding full generator. The differentiation lemma for parameter-dependent integrals entails that $T(t) f$ is differentiable for any $t>0$, $f \in L^{\infty}(\R^d)$, and
	\begin{equation*}
		\|\nabla T(t) f\|_{\infty} \leq \varphi(t) \|f\|_{\infty} \quad \text{with} \quad \varphi(t) := \max\left\{1,\int_{\R^d} |\nabla p_t(x)| \, dx\right\}.
	\end{equation*}
	As $p_{2t}$ is the convolution of $p_t$ with itself,  $\|\nabla^2 T(t) f\|_{\infty} \leq d^2 \varphi(t/2)^2 \|f\|_{\infty}$, cf.\ \cite[Lem.\ 4.1]{euler-maruyama}. Using that $\Cb^{\gamma}(\R^d)$, $\gamma \in (0,2) \setminus \{1\}$, is a real interpolation space between $\Cb(\R^d)$ and $\Cb^2(\R^d)$, see \cite[Thm.\ 2.7.2.1]{triebel78}, and applying an interpolation theorem \cite[Sec.\ 1.3.3]{triebel78}, we find that $\|T(t) f\|_{\Cb^{\gamma}(\R^d)} \leq M \varphi(t/2)^{\gamma} \|f\|_{\infty}$ for some constant $M=M(\gamma)>0$ and any $f \in L^{\infty}$, $\gamma \in [0,2]$. 
\end{proof} 

Gradient estimates of the form \eqref{eq.gradientestimate} for L\'evy processes have been studied intensively in the last decade, e.g.\ \cite{szcz17,kaleta15,lm6,ryznar15,ssw} to mention but a few, and so Proposition~\ref{p.levy} applies to a wide class of L\'evy generators. Using the gradient estimates from \cite[Ex.\ 1.5]{ssw}, we find that the assumptions of Proposition~\ref{p.levy} are satisfied for L\'evy measures which are dominated from below by a stable measure.

\begin{example} \label{ex.frac.laplacian}
	Let $(T(t))_{t \geq 0}$ be the semigroup of a pure-jump L\'evy process with L\'evy measure $\nu$ satisfying
	\begin{equation*}
		\nu(U) \geq \int_0^{r_0} \int_{\mathbb{S}} \I_U(r \theta) r^{-1-\alpha} \, dr \, \varrho(d\theta), \quad U \in B(\R^d \setminus \{0\}),
	\end{equation*}
	for some constants $\alpha \in (0,2)$, $r_0>0$ and a measure $\varrho$ on the unit sphere $\mathbb{S}$ whose support is not contained in any proper linear subspace of $\R^d$. Then the assumptions of Proposition~\ref{p.levy} are satisfied (with $\kappa<\alpha$).
\end{example}

Example~\ref{ex.frac.laplacian} covers, in particular,  stable operators, i.e.\ operators \eqref{eq.generalop} with $b=0$, $Q=0$ and L\'evy measures $\nu$ of the form $\nu(U) = \int_0^{\infty} \int_{\mathbb{S}} \I_U(r \theta) r^{-1-\alpha} \, dr \, \varrho(d\theta)$. Let us mention that Komatsu \cite{komatsu84} was one of the first to study perturbations of stable operators under L\'evy-type operators; however, his result requires quite strong assumptions on the regularity the density of $\nu$. Moreover, we would like to mention a closely related result by Peng \cite{peng} who considers perturbation of the fractional Laplacian by  time-dependent kernels $\mu(t,x,dy)$. Our approach has the advantage that it is not restricted to stable operators but works in a much more general setting, as demonstrated by Proposition \ref{p.levy}

\subsection{Operators of variable order} \label{s.varorder}

Our next application concerns perturbations of the fractional Laplacian of variable order $-(-\Delta)^{\alpha(\cdot)/2}$, which is defined on $\Cc^{\infty}(\R^d)$ by
\begin{equation*}
	-(-\Delta)^{\alpha(\cdot)/2} f(x) = c_{d,\alpha(x)} \int_{\{y \neq 0\}} (f(x+y)-f(x)-y \cdot \nabla f(x) \I_{(0,1)}(|y|)) \frac{1}{|y|^{d+\alpha(x)}} \, dy
\end{equation*}
for a certain normalizing constant $c_{d,\alpha(x)}$. This operator appears as generator of so-called processes of variable order, see e.g.\ \cite{bass,kol,lm6} for the construction of such processes and more information about the probabilistic background. Denote by $(T(t))_{t \geq 0}$ the associated $\Co$-Feller semigroup. 

\begin{prop} \label{p.variableorder}
	If $\alpha: \R^d \to (0,2]$ is H\"older continuous with $\alpha_L := \inf_x \alpha(x)>0$,  then $(T(t))_{t \geq0}$ satisfies Hypothesis~\ref{h1b} and \ref{hyp.mp} with $\rho=\alpha_L$, and so Theorem~\ref{t.generic} is applicable (with $\rho=\alpha_L$).
\end{prop}

\begin{proof}
	Let $(T(t))_{t \geq 0}$ be the Markovian Feller semigroup associated with the operator $-(-\Delta)^{\alpha(\cdot)/2}$, see e.g.\ \cite{bass} or \cite[Thm.\ 5.2]{lm6} for the existence, and denote by $\hat{A}$ its full generator and by $A$ its part in $\Co(\R^d)$. It is known from the construction of the semigroup that $\Cc^{\infty}(\R^d)$ is a core and that $T(t) f(x) = \int f(y) p_t(x,y) \, dy$ for a bounded transition function $p_t$, cf.\ \cite{lm6}. In consequence, by \cite[Thm.\ 1.9, Thm.\ 1.14]{ltp3}, each $T(t)$ is a strong Feller operator. Moreover, by \cite[Prop.\ 6.1]{reg-feller}, 
	\begin{equation*}
		\|T(t) f\|_{\Cb^{\beta}(\R^d)} \leq c_{\beta} t^{-\beta/\alpha_L} \|f\|_{\infty}, \qquad f \in \Bb(\R^d),\;t \in (0,1),
	\end{equation*}
	for any $\beta<\alpha_L$. In conclusion, Hypothesis \ref{hyp.mp} is satisfied.
\end{proof}

\begin{example}
Consider
\[
\cL u = -(-\Delta)^{\alpha(\cdot)/2}u - (-\Delta)^{\beta(\cdot)/2}
\]
where $\alpha: \R^d\to (0,2]$ is H\"older continuous with $\alpha_L \coloneqq \inf_x \alpha (x) > 0$ and $\beta: \R^d\to (0,2)$ is a measurable function with $\sup_x \beta(x) < \alpha_L$. Then, by Proposition~\ref{p.variableorder} and Theorem~\ref{t.generic}, the bp-closure of $\{(f,\cL f); f \in \Cc^{\infty}(\R^d)\}$ is the generator of a Markovian $\Cb$-Feller semigroup $S$ and the $(\cL,\Cc^{\infty}(\R^d))$-martingale problem is well-posed. If additionally $\inf_x \beta(x)>0$, then  the tightness condition \eqref{e.tightness} is satisfied and $S$ is a $\Co$-Feller semigroup.

Let us mention that Theorem \ref{t.sg.conv} also yields continuous dependence of these solutions on the `coefficient' $\beta$. More precisely, 
if $\beta_n$, $n \in \N$ are measurable functions such that $\beta_n(x) \to \beta(x)$ for all $x \in \R^d$ and  $0<\inf_n \inf_x \beta(x) \leq \sup_n \sup_x \beta_n(x)<\alpha_L$, then $\PP^{(\beta_n,x)} \to \PP^{(\beta,x)}$ weakly for every $x \in \R^d$. In particular, the associated semigroups satisfy $S^{(\beta_n)}(t)f(x) \to S^{(\beta)}(t)f(x)$ for any $t>0$, $x\in \R^d$ and $f \in \Cb(\R^d)$.
\end{example}

Proposition~\ref{p.variableorder} can be extended to a wider class of non-local operators, that is, we can replace the fractional Laplacian of variable order $-(-\Delta)^{\alpha(\cdot)/2}$ by other operators of variable order, see \cite[Section 5.1]{lm6} for some interesting choices.

\subsection{Non-local operators and L\'evy-driven SDEs} \label{s.sde}

In this subsection, we show how our perturbation results can be applied in the theory of stochastic differential equations (SDEs). The key is a result by Kurtz \cite{kurtz} which states that the existence of solutions to L\'evy-driven SDEs can be equivalently characterized by the existence of solutions to martingale problems associated with certain L\'evy-type operators \eqref{eq.generalop}. For a comprehensive study of L\'evy-driven SDEs, e.g.\ classical existence and uniqueness results, we refer to Ikeda--Watanabe \cite{ikeda} and Protter \cite{protter}. Our first result concerns SDEs with additive L\'evy noise.

\begin{prop} \label{p.levy-drift}
	Let $(L_t)_{t \geq 0}$ be a L\'evy process with characteristic exponent $\psi$ satisfying the Hartman--Wintner condition \eqref{hw}. If the transition density $p_t$ satisfies $\int_0^1 \int_{\R^d} |\nabla p_t(x)| \ dx \, dt < \infty$, then the SDE
	\begin{equation}
		dX_t = b(X_t) \, dt + dL_t, \qquad X_0=x, \label{sde-levy}
	\end{equation} 
	has a unique weak solution. Moreover, the unique solution does not explode in finite time and gives rise to a Markovian Feller semigroup with symbol 
	\begin{equation*}
		q(x,\xi) = -i b(x) \cdot \xi + \psi(\xi), \qquad x,\xi \in \R^d.
	\end{equation*} 
\end{prop}

\begin{proof}
This is immediate from Proposition \ref{p.levy} (with $\kappa=1$) and the results of \cite{kurtz}. 
\end{proof}

Proposition~\ref{p.levy-drift} applies, in particular, to any L\'evy process with characteristic exponent $\psi$ satisfying $\Re \psi(\xi) \geq c |\xi|^{\alpha}$, $|\xi| \gg 1$, for some $\alpha>1$; indeed, \eqref{hw} is then trivial and the integrability condition on $\nabla p_t$ follows from \cite[Thm.\ 3.2]{ssw}. Let us mention that Tanaka et.\ al \cite[pp.~82,83]{tanaka} obtained the existence of a unique solution under the milder assumption that $\frac{1}{\Re \psi(\xi)} = o\left(|\xi|^{-1}\right)$ as $|\xi| \to \infty$. Furthermore, there is ongoing research about the optimal assumptions on the drift $b$ to ensure the existence of a unique weak solution to \eqref{sde-levy} for a given L\'evy process; e.g.\  \cite{chen15} considers isotropic $\alpha$-stable drivers, $\alpha>1$, and (generally unbounded) drifts $b$ from some Kato class.

\begin{prop} \label{p.sde-drift}
	Let $(L_t)_{t \geq 0}$ be a one-dimensional isotropic $\alpha$-stable L\'evy process for $\alpha \in (1,2)$, i.e.\ $(L_t)_{t \geq 0}$ is a L\'evy process with characteristic exponent $\psi(\xi)=|\xi|^{\alpha}$. Let $\sigma: \R \to \R$ be a H\"older continuous function with $0<\inf_x \sigma(x) \leq \sup_x \sigma(x)<\infty$. For every $b \in \Bb(\R)$, there is a unique weak solution to the SDE
	\begin{equation}
		dX_t = b(X_{t-}) \, dt + \sigma(X_{t-}) \, dL_t, \qquad X_0 = x. \label{eq.sde-drift}
	\end{equation}
	The unique solution does not explode in finite time and gives rise to a Markovian $\Co$-Feller semigroup with symbol
	\begin{equation*}
		q(x,\xi) = -ib(x) \cdot \xi + |\sigma(x)|^{\alpha} |\xi|^{\alpha}.
	\end{equation*}
\end{prop}

\begin{proof}
	Since $\sigma$ is H\"older continuous, the SDE $dY_t = \sigma(Y_{t-}) \, dL_t$ has a unique weak solution which gives rise to a Feller process, see e.g. \cite[Thm.\ 2.1]{kulik19} or \cite[Thm.\ 5.23]{lm6}. Denote by $(T(t))_{t \geq 0}$ the associated semigroup and by $\hat{A}$ the full generator. It is known that $\Cc^{\infty}(\R)$ is a core for the generator of $T|_{\Co(\R)}$, cf.\ \cite{kulik19,lm6}. The strong Feller property of $T(t)$ follows from \cite[Thm.\ 1.14]{ltp3} and the fact that $Y_t$ has a bounded transition density $p_t$, see e.g. \cite{lm6}.  Moreover, \cite[(Proof of) Proposition 4.5]{reg-feller} yields the regularity estimate
	\begin{equation}
		\|T(t) f\|_{\Cb^{\beta}(\R)} \leq c_{\beta} t^{-\beta/\alpha} \|f\|_{\infty}, \qquad f \in \Bb(\R),\;t>0, \label{eq.reg-sde}
	\end{equation}
	for $\beta<\alpha$. Choosing $\beta \in (1,\alpha)$ and setting $\hat{B}f(x)=b(x) \nabla f(x)$, we see that the assumptions of Theorem~\ref{t.perturbeuclidean} are satisfied, and so there is a Markovian $\Co$-semigroup with full generator $\hat{A}+\hat{B}$. Using the well-posedness of the $(A+\hat{B},\Cc^{\infty}(\R))$-martingale problem, which holds by Theorem~\ref{t.mp.wellposed}, it follows from \cite{kurtz} that the SDE \eqref{eq.sde-drift} has a unique weak solution.
\end{proof}

\begin{rem} \label{r.sde-drift} 
Proposition \ref{p.sde-drift} can be extended in several ways:
\begin{enumerate}
	\item We can consider a wider class of driving L\'evy processes $(L_t)_{t \geq 0}$. There are two key ingredients which we need. Firstly, the existence of a unique weak solution to the SDE $dY_t = \sigma(Y_{t-}) \, dL_t$; there is extensive literature on this topic, see e.g.\ \cite{chen,kul21,lm6,kulik19,zan02} and the references therein. Secondly, a regularity estimate $\|T(t) f\|_{\Cb^{\beta}(\R^d)} \leq \varphi(t) \|f\|_{\infty}$ for the corresponding semigroup, where $\varphi \in L^1(0,1)$ and $\beta>1$. Though there are quite some works on regularizing properties of semigroups associated with SDEs, many of them study only $\beta$-H\"older regularity for $\beta \leq 1$, and this is not enough for our purposes.
	\item Using localization techniques for martingale problems, one can relax the growth assumptions on $b$ and $\sigma$; see \cite[Thm.\ 4.2]{mp-feller} and \cite[Thm.\ 5.3]{hoh} for useful results in this direction.
\end{enumerate}
\end{rem}

\begin{prop} \label{p.sde-jumps} 
	Let $(L_t)_{t \geq 0}$ be a one-dimensional isotropic $\alpha$-stable L\'evy process for $\alpha \in (0,2)$. Let $(M_t)_{t \geq 0}$ be a one-dimensional L\'evy process independent of $(L_t)_{t \geq 0}$ with L\'evy triplet $(b,0,\nu)$ such that $\int_{\{|y| \leq 1\}} |y|^{\beta} \, \nu(dy) < \infty$ for some $\beta \in [0,\alpha)$. If $\alpha \leq 1$ also assume that the compensated drift $b-\int_{\{|y|<1\}} y \, \nu(dy)$ is zero. Let $\sigma: \R \to (0,\infty)$ be a H\"older continuous mapping with $0<\inf_x \sigma(x) \leq \sup_x \sigma(x)<\infty$. For every $\kappa \in \Bb(\R)$, the SDE
	\begin{equation*}
		dX_{t-} = \kappa(X_{t-}) \, dM_t + \sigma(X_{t-}) \, dL_t, \qquad X_0=x
	\end{equation*}
	has a unique weak solution. It does not explode in finite time and it gives rise to a Markovian Feller semigroup with symbol
	\begin{equation*}
		q(x,\xi) = \psi(\kappa(x) \xi) + |\sigma(x)|^{\alpha} |\xi|^{\alpha},
	\end{equation*}
	where $\psi$ denotes the characteristic exponent of $(M_t)_{t \geq 0}$.
\end{prop}

Proposition~\ref{p.sde-jumps} can be extended to other driving L\'evy processes $(L_t)_{t \geq 0}$, cf.\ Remark~\ref{r.sde-drift}.

\begin{proof}[Proof of Proposition~\ref{p.sde-jumps}]
	Because of the integrability condition on the L\'evy measure $\nu$, an application of Taylor's formula shows that the operator 
	\begin{align*}
		\hat{B}f(x) & \coloneqq  b \cdot \nabla f(x) \I_{\{\alpha >1\}}\\
		& \quad  +\int_{\{y \neq 0\}} (f(x+\kappa(x) y)-f(x)- f'(x) \kappa(x) y \I_{(0,1)}(|y|) \I_{\{\alpha>1\}}) \, \nu(dy)
	\end{align*}
	satisfies $\|\hat{B}f\|_{\infty} \leq C \|f\|_{\Cb^{\beta}(\R)}$ for some constant $C>0$. Moreover, Hypothesis~\ref{h2} holds. Note that the kernels $\mu(x,dy)$ associated with $\hat{B}$ satisfy the tightness condition \eqref{e.tightness} because
	\begin{align*}
		\mu(x,\{|y|\geq R\}) 
		& = \nu(\{y \in \R \setminus \{0\}; |\kappa(x) y| \geq R\})\\
		&\leq \nu(\{y \in \R \setminus \{0\}; |y| \geq R/\|\kappa\|_{\infty}\}) \xrightarrow[]{R \to \infty} 0.
	\end{align*}
	 Now we can proceed exactly as in the proof of Proposition~\ref{p.sde-drift}; the only difference is that we consider the just-defined operator $\hat{B}$ rather than the drift operator.
\end{proof}

\begin{example}
	Let $\alpha \in (0,2]$ and $\beta \in (0,\alpha)$. Let $(L_t)_{t \geq 0}$ be an isotropic $\alpha$-stable L\'evy process and $(M_t)_{t \geq 0}$ an isotropic $\beta$-stable L\'evy process, which is independent of $(L_t)_{t \geq 0}$. Let $\sigma$ be a H\"older continuous function with $0<\inf_x \sigma(x) \leq \sup_x \sigma(x)<\infty$. Then, by Proposition~\ref{p.sde-jumps} and Theorem~\ref{t.sg.conv}: \begin{enumerate}
		\item For every $\kappa \in \Bb(\R^d)$, the SDE
		\begin{equation}
			dX_t = \kappa(X_{t-}) \, dM_t + \sigma(X_{t-}) \, dL_t, \qquad X_0 = x,  \label{eq.sde-jumps}
		\end{equation}
		has a unique weak solution. It gives rise to a Markovian Feller semigroup.
		\item If $\kappa_n \in \Bb(\R)$, $n \in \N$, is a sequence of measurable functions such that $\kappa_n \bp \kappa$, then the solutions $X^{(\kappa_n)}$ to \eqref{eq.sde-jumps} (with $\kappa$ replaced by $\kappa_n$) converge weakly to $X=X^{(\kappa)}$ in the Skorohod space $\cD[0,\infty)$.
	\end{enumerate}
\end{example}

\appendix

\section{bp-convergence of measure-valued functions}\label{appa}

The concept of \emph{bp-convergence}, see Section~\ref{kernel}, is a very useful tool in probability theory and allows for a version of Dynkin's $\pi$-$\lambda$ theorem for bounded measurable functions. We recall that if $E$ is a metric space, then a set $M \subset \Bb(E)$ is called \emph{bp-closed} if whenever $(f_n) \subset M$ is a sequence with $f_n \bp f$, then also $f\in M$. The \emph{bp-closure} of $M\subset \Bb(E)$ is the smallest bp-closed subset of $\Bb(E)$ that contains $M$. If the bp-closure of $M$ is $\Bb(E)$ we say that $M$ is \emph{bp-dense} in $\Bb(E)$.

We recall the following general result from \cite[Prop.\ 3.4.2]{ek}.

\begin{prop}\label{p.bpdense}
In a metric space $E$, the set $\Cb(E)$ is bp-dense in $\Bb(E)$. If $E$ is additionally locally compact, then also $\Cc(E)$ is bp-dense in $\Bb(E)$.
\end{prop}

In this appendix, we aim to extend this result to functions that take values in the cone $\Mp(E)$ of positive, locally finite measures on $E$, where $E$ is a locally compact, separable metric space. To that end, we endow $\Mp(E)$ with the vague topology, induced by the space $\Cc(E)$ of continuous functions with compact support. As is well known, see \cite[Thm.\ A.2.3(i)]{kallenberg}, $\Mp(E)$ is a Polish space in the vague topology, i.e.\ the topology is induced by a complete, separable metric. 

As a matter of fact, we will not be interested in the case of arbitrary $E$, but only in the case $E=\R^d \setminus \{0\}$. Fix a continuous function $g: \R^d \setminus\{0\} \to [0,\infty)$. Given a measure $\nu \in\Mp(\R^d\setminus\{0\})$, we put
\[
|\nu|_{g} \coloneqq \int_{\R^d\setminus\{0\}} g(y) \, d\nu (y).
\] 
We denote by $K^g (\R^d,  \Mp(\R^d\setminus\{0\}))$ the set of measurable mappings $\mu : \R^d\to \Mp(\R^d\setminus \{0\})$ such that 
\begin{align*}
 \|\mu\|_g \coloneqq \sup_{x\in \R^d}|\mu(x)|_g < \infty
  \quad \mbox{ and } \quad \lim_{R\to\infty} \sup_{x\in K} \mu (x, \{ |y|> R\}) = 0\big\} \, \, \text{for all $K \Subset \R^d$}.
\end{align*}
We note that, as a consequence of \cite[Thm.\ A.2.3(iv)]{kallenberg} the map $\mu : \R^d \to \Mp(\R^d\setminus\{0\})$ is measurable if and only if $x\mapsto \int_{\R^d} f(y) \mu (x, dy)$ is measurable for every $f\in \Cc(\R^d\setminus\{0\})$ if and only if $x\mapsto \mu (x, A)$ is measurable for every relatively compact set $A$.

\begin{defn}\label{def.bp}
We say that a sequence $(\mu_n) \subset K^g (\R^d, \Mp(\R^d\setminus\{0\}))$ \emph{converges boundedly and pointwise} to $\mu: \R^d \to \Mp(\R^d\setminus\{0\})$, if
\begin{enumerate}
\item $\sup_n\|\mu_n\|_g <\infty$,
\item $\mu_n(x) \to \nu (x)$ vaguely for every $x\in \R^d$ and
\item for every $K\Subset \R^d$ we have $\lim_{R\to\infty} \sup_{n\in \N} \sup_{x\in K} \mu_n (x, \{|y|> R\}) =0$.
\end{enumerate}
We write $\mu_n \stackrel{bp}{\to}\nu$ to indicate bounded and pointwise convergence.
\end{defn}

\begin{rem}\label{rem.automatic}
We note that if $g(x) = \varphi (|x|)$ for some function $\varphi : (0,\infty) \to (0,\infty)$ with $\varphi (t) \uparrow \infty$ 
as $t\to \infty$, then condition (iii) in Definition \ref{def.bp} is automatically fulfilled and follows from 
assumption (i). Indeed, for $R>0$, $x\in \R^d$ and $n\in \N$ we have
\[
\mu_n(x, \{|y|>R\}) \leq \frac{1}{\varphi (R)} \int_{\{|y|>R\}} \varphi (|y|)\, \mu_n(x, dy) \leq 
\frac{1}{\varphi(R)} \sup_{n\in \N} \|\mu_n\|_g \to 0.
\]
\end{rem}

We note that we did not assume in the above definition that the limit $\mu$ belongs to $K^g (\R^d, \Mp(\R^d\setminus\{0\}))$. As it turns out, this holds true automatically.

\begin{lem}\label{l.kalphabpclosed}
Let $(\mu_n) \subset K^g (\R^d, \Mp(\R^d\setminus\{0\}))$. If $\mu_n \bp\mu$, then also
$\mu$ belongs to $K^g (\R^d, \Mp(\R^d\setminus\{0\}))$. 
\end{lem}

\begin{proof}
We first note that $\mu$ is measurable as pointwise limit of measurable functions. We now pick $C>0$ such that
$\|\mu_n\|_g \leq C$ for all $n\in \N$ and fix a sequence $(f_k)\in \Cc(\R^d\setminus\{0\})$ such that
$0\leq f_k(y) \uparrow g(y)$. As $\mu_n(x) \to \mu (x)$ vaguely, we obtain for  fixed $k\in \N$ and arbitrary $x\in \R^d$ that
\begin{align*}
\int_{\R^d} f_k(y) \mu (x, dy)  & = \lim_{n\to\infty} \int_{\R^d} f_k(y) \mu_n(x, dy)
 \leq \limsup_{n\to\infty} \int_{\R^d}
g(y)\, \mu_n(x, dy) \leq C.
\end{align*}
Taking the limit $k\to \infty$, monotone convergence implies 
\[
\int_{\R^d}
g(y)\, \mu(x, dy) \leq C
\]
for every $x\in \R^d$, proving that $\|\mu\|_g \leq C<\infty$. To prove that for a compact set $K$, we have
$\lim_{R\to \infty}\sup_{x\in K}\mu (x, \{|y|>R\})=0$, we pick, given $\eps>0$, a radius $R>0$ such that
$\mu_n(x, \{|y|>R\}) \leq \eps$ for all $x\in K$ and $n\in \N$. Approximating $\one_{\{|y|>R\}}$ by an increasing sequence of functions
in $\Cc(\R^d)$, we can repeat the above argument to prove that $\mu (x, \{|y| >R\}) \leq \eps$ for all $x\in K$.
\end{proof}

In what follows, we use the notions of bp-closedness, bp-closure and bp-density with the obvious meaning also for subsets of $K^g(\R^d, \Mp(\R^d\setminus\{0\}))$. We can now prove the main result of this appendix, a version of Proposition \ref{p.bpdense} for measure-valued functions.

\begin{thm}\label{t.bpdense}
The set $\Cc(\R^d, \Mp(\R^d\setminus\{0\})) \cap K^g(\R^d, \Mp(\R^d\setminus\{0\}))$ is bp-dense in $K^g(\R^d, \Mp(\R^d\setminus\{0\}))$.
\end{thm}

\begin{proof}
We write $K^g_C$ for the set which is claimed to be bp-dense and denote by $\mathcal{F}$ its bp-closure. We proceed in several steps.\smallskip

\emph{Step 1}: We prove that if $\mu(x, A) = \varphi (x) \cdot \nu(A)$ for some $\varphi \in \Bb(\R^d)$ and $\nu \in \Mp(\R^d\setminus \{0\})$ with $|\nu|_g <\infty$, then $\mu \in \mathcal{F}$.

Indeed, if $\varphi \in \Cc(\R^d)$, then a function $\mu$ of the above form belongs to $K^g_C$ and hence to $\mathcal{F}$. However, if $\varphi_n \bp\varphi$ it is easy to see that $\varphi_n(\cdot)\nu\bp\varphi(\cdot)\nu$ and hence the latter belongs $\mathcal{F}$ if this is true for the approximating sequence. Invoking Proposition \ref{p.bpdense}, it follows that indeed for every $\varphi \in \Bb(\R^d)$ the above function $\mu$ belongs to $\mathcal{F}$.\medskip

\emph{Step 2}: We prove that $\mathcal{F}$ is a cone.

To see this, fix $\mu \in \mathcal{F}$ and put 
\[
\mathcal{F}_\mu \coloneqq\big\{ \tilde\mu \in \mathcal{F} : s\mu + t\tilde \mu \in \mathcal{F} \mbox{ for all } s, t \in (0,\infty)\}.
\]
A moments thought shows that $\mathcal{F}_\mu$ is bp-closed. If $\mu \in K^g_C$, then clearly every element of $K^g_C$ belongs to $\mathcal{F}_\mu$ as $K^g_C$ is itself a cone. Thus $\mathcal{F}\subset \mathcal{F}_\mu$ and hence $\mathcal{F}=\mathcal{F}_\mu$. 
Now let $\mu \in \mathcal{F}$ be arbitrary. Then we have $K^g_C \subset \mathcal{F}_\mu$ by what was just proved. Using that $K^g_C$ is bp-dense in $\mathcal{F}$ once again, we find $\mathcal{F}_\mu= \mathcal{F}$ for every $\mu \in \mathcal{F}$ which proves that $\mathcal{F}$ is indeed a cone.\medskip

\emph{Step 3}: It follows from Step 1 and Step 2 that every function of the form
\[
\mu (x,\cdot ) = \sum_{k=1}^n \one_{A_k}(x)\nu_k(\cdot),
\]
where $A_1, \ldots, A_n \in \mathcal{B}(\R^d)$ and $\nu_1, \ldots, \nu_n \in \Mp(\R^d\setminus\{0\})$ with $|\nu_k|_g <\infty$ for all $k=1, \ldots, n$, belongs to $\mathcal{F}$. Given any sequence $(A_k) \subset \mathcal{B}(\R^d)$ of pairwise disjoint sets and a tight sequence $\nu_k \in \Mp(\R^d\setminus\{0\})$ with $\sup_n|\nu_n|_g <\infty$ also the function
\[
\mu (x, \cdot) \coloneqq \sum_{k=1}^\infty \one_{A_k}(x)\nu_k(\cdot)
\]
belongs to $\mathcal{F}$; indeed, the partial sums are all in $\mathcal{F}$ and so their bp-limit $\mu$ belongs to $\mathcal{F}$ by the closedness of $\mathcal{F}$.  \medskip

\emph{Step 4} 
Fix $\mu\in K^g(\R^d, \Mp(\R^d\setminus\{0\}))$ with $\lim_{R\to \infty} \sup_{x\in \R^d} \mu (x, \{|y|>R\}) = 0$. We put
\begin{align*}
S_\mu & \coloneqq \{ \nu \in \Mp(\R^d\setminus\{0\}) : |\nu|_g \leq \|\mu\|_g \mbox{ and }\\
& \qquad  \nu (\{|y|> R\}) \leq \sup_{x\in \R^d} \mu (x, \{|y|> R\}) \mbox{ for all } R> 0\}.
\end{align*}
Arguing similar as in the proof of Lemma \ref{l.kalphabpclosed}, we see that $S_\mu$ is a closed subset of $\Mp(\R^d\setminus\{0\})$.
Consequently, $S_\mu$ is separable, whence we find a dense sequence $(\nu_n) \subset S_\mu$ that is dense in $S_\mu$. 

Let us fix a metric $\rho$ that generates the topology on $S_\mu$. We write
$B(\nu, \eps)$ for the open ball in $S_\mu$ of radius $\eps$ and with center $\nu\in S_\mu$. We can now inductively define sets
$B_k^{(n)}$ by $B_1^{(n)} = B(\nu_1, n^{-1})$ and
\[
B_{k+1}^{(n)} \coloneqq B(\nu_{k+1}, n^{-1})\setminus (B_1^{(n)}\cup \cdots\cup B_k^{(n)}).
\]
Then for every $n\in \N$ the sets $(B_k^{(n)})_{k\in\N}$ are pairwise disjoint and their union equals $S_\mu$. Let us set
\[
\mu_n (x, \cdot) \coloneqq \sum_{k=1}^\infty \one_{\mu^{-1}(B_j^{(n)})}(x)\nu_j(\cdot).
\]
By Step 3, the function $\mu_n$ belongs to $\mathcal{F}$. However, we have $\|\mu_n \|_g \leq \|\mu\|_g$ for every $n\in \N$ as every $\nu_j$ belongs to $S_\mu$ and for $x\in \R^d$ we have $\rho (\mu_n(x), \mu (x)) \leq n^{-1}$ so that $\mu_n(x) \to \mu (x)$. By the definition of the set $S_\mu$ also condition (iii) in Definition \ref{def.bp} is satisfied so that $\mu_n \bp\mu$. It follows that $\mu \in \mathcal{F}$.\medskip

\emph{Step 5} We finish the proof.

Given $\mu \in K^g(\R^d, \Mp(\R^d\setminus\{0\}))$, we set $\mu_n \coloneqq \one_{B(0,n)}\mu$. By Step 4, we have $\mu_n \in \mathcal{F}$, as $\sup_{x\in \R^d}\mu_n(x, \{|y|>R\})\to 0$ as $R\to \infty$. On the other hand, we have $\mu_n \bp \mu$ and it follows that $\mu \in \mathcal{F}$.
\end{proof}

We now combine Proposition \ref{p.bpdense} and Theorem \ref{t.bpdense} into a single result which will be useful in the main part of this article.

\begin{cor}\label{c.bpdense}
Assume that $F$ is a subset of $\Bb(\R^d; \R^d)\times K^g(\R^d, \mathcal{M}^+(\R^d\setminus\{0\}))$ such that
\begin{enumerate}
[(i)]
\item $\Cc(\R^d;\R^d)\times\Big[\Cc\big(\R^d, \mathcal{M}^+(\R^d\setminus \{0\})\big) \cap K^g\big(\R^d, \mathcal{M}^+(\R^d\setminus\{0\})\big)\big]\subset F$
\item whenever $((f_n, \mu_n))_{n\in \N} \subset F$ and $f_n \stackrel{\mathrm{bp}}{\to} f$ and
$\mu_n \stackrel{\mathrm{bp}}{\to} \mu$, then also $(f, \mu)\in F$.
\end{enumerate}
Then $F= \Bb(\R^d)\times K^g(\R^d, \mathcal{M}^+(\R^d\setminus\{0\}))$.
\end{cor}

\begin{proof}
Fixing $f\in \Cc(\R^d; \R^d)$ and considering sequences of the form $(f, \mu_n)_{n\in \N}$ bp-converging to $(f,\mu)$, Theorem \ref{t.bpdense} shows $\Cc(\R^d)\times K^g (\R^d, \mathcal{M}^+(\R^d\setminus\{0\})) \subset F$. Let us now fix $\mu \in K^g (\R^d, \mathcal{M}^+(\R^d\setminus\{0\}))$ and consider sequences $(f_n, \mu)$ bp-converging to $(f, \mu)$. Then Proposition \ref{p.bpdense} yields the claim.
\end{proof}

\section{Continuity of L\'evy-type operators in terms of the symbol and the characteristics} \label{appb}

In the proof of one of our perturbation results, Theorem~\ref{t.perturbeuclidean}, we used that the integro-differential operator
\begin{equation}
	\hat{B}f(x) = b(x) \cdot \nabla f(x) + \int_{\R^d \setminus \{0\}} (f(x+y)-f(x)-y \cdot \nabla f(x) \chi(y)) \, \nu(x,dy)
	\label{eq-intdiff2}
\end{equation}
maps (sufficiently) smooth functions to continuous functions. On $\Cc^{\infty}(\R^d)$, the operator can be considered as a pseudo-differential operator with symbol $q(x,\xi)$, cf.\ \eqref{eq.pseudo}, and so one can equivalently asked for conditions on the symbol ensuring that $\hat{B}f \in \mathcal{C}(\R^d)$ for $f \in \Cc^{\infty}(\R^d)$. The following result answers this question; it is somewhat more refined than what we need for our proof but we believe the result to be of independent interest.

\begin{thm} \label{appb-1}
	Let $\hat{B}$ be as in \eqref{eq-intdiff2} and denote by \begin{equation}
		q(x,\xi) = -i b(x) \cdot \xi + \int_{\R^d \setminus \{0\}} (1-e^{iy \cdot \xi}+iy \cdot \xi \chi(y)) \, \nu(x,dy) \label{appb-eq3}
	\end{equation}
	the associated symbol; here $\chi \in \Cc^{\infty}(\R^d)$ is a smooth cut-off function with $\I_{B(0,1)} \leq \chi \leq \I_{B(0,2)}$. If $q$ is locally bounded, then the following statements are equivalent. \begin{enumerate}
		\item\label{appb-1-iii} $\hat{B}f$ is continuous for every $f \in \Cc^{\infty}(\R^d)$. 
		\item\label{appb-1-i} $x \mapsto q(x,\xi)$ is continuous for all $\xi \in \R^d$.
		\item\label{appb-1-ii} Each of the following conditions is satisfied. \begin{enumerate}
		\item $x \mapsto b(x)$ is continuous.
		\item $x \mapsto \nu(x,\cdot)$ is vaguely continuous on $(\R^d \backslash \{0\},\cB(\R^d \backslash \{0\}))$.
		\item The family $(\nu(x,\cdot))_{x \in K}$, is tight for any compact set $K \Subset \R^d$, i.e.\ $\lim_{R \to \infty} \sup_{x \in K} \nu(x, \{|y| \geq R\}) = 0$.
		\item $\lim_{r \to 0} \sup_{x \in K} \int_{\{|y| \leq r\}} |y|^2 \, \nu(x,dy)=0$ for any compact set $K \Subset \R^d$.
		\end{enumerate}
	\end{enumerate}
	If one (hence, all) of the conditions is satisfied and $\hat{B}: \Cb^{\beta}(\R^d) \to \Bb(\R^d)$ is a bounded operator for some $\beta>0$, then $\hat{B}f$ is continuous for all $f \in \Cb^{\beta}(\R^d) \cap \Co(\R^d)$.
\end{thm}

\begin{rem} \label{appb-2} \begin{enumerate}[(a)]
	\item The implications \ref{appb-1-i} $\iff$  \ref{appb-1-iii} $\implies$ \ref{appb-1-ii} remain valid for any Borel measurable cut-off function $\chi$ such that $\I_{B(0,1)} \leq \chi \leq \I_{B(0,2)}$.
	\item A symbol $q$ of the form \eqref{appb-eq3} is locally bounded if, and only if, for any compact set $K \subseteq \R^d$ there exists a constant $c>0$ such that $|q(x,\xi)| \leq c(1+|\xi|^2)$ for all $x \in K$, $\xi \in \R^d$. By \cite[Lem.\ 2.1, Rem.\ 2.2]{rs-grow}, this is equivalent to \begin{equation}
		\forall K \Subset \R^d \, \, :\: \sup_{x \in K} |b(x)| + \sup_{x \in K} \int_{\R^d \setminus \{0\}} \min\{|y|^2, 1\} \, \nu(x,dy) <\infty.  \label{appb-eq5}
	\end{equation}
	\end{enumerate}
\end{rem}

\begin{proof}[Proof of Theorem~\ref{appb-1}]
	To keep notation simple, we prove the result only in dimension $d=1$. The implication \ref{appb-1-iii} $\implies$ \ref{appb-1-i} follows from \cite[Thm.\ 4.4]{rs98}. Moreover, if $x \mapsto q(x,\xi)$ is continuous, then we find from the local boundedness of $q$ and the dominated convergence theorem that \begin{equation*}
	  x \mapsto \hat{B}f(x) = - \int_{\R} q(x,\xi) \hat{f}(\xi) e^{ix \xi} \, d\xi
	\end{equation*}
	is continuous for all $f \in \Cc^{\infty}(\R)$, and this proves \ref{appb-1-i} $\implies$ \ref{appb-1-iii}. In the remainder of the proof we show that \ref{appb-1-i} $\iff$ \ref{appb-1-ii}. \par \medskip
	
	\ref{appb-1-ii} $\implies$ \ref{appb-1-i}: By (iii)(a), it suffices to show that \begin{equation*}
		x \mapsto p(x,\xi): = \int_{\R \setminus \{0\}} (1-e^{iy \xi} + i y \xi \chi(y)) \, \nu(x,dy)
	\end{equation*}
	is continuous for all $\xi \in \R$. Clearly, \begin{equation*}
		|p(x,\xi)-p(z,\xi)| \leq I_1 + I_2+I_3,
	\end{equation*}
	where \begin{align*}
		I_1 &:= \frac{|\xi|^2}{2} \left( \int_{\{|y| \leq r\}} |y|^2 \, \nu(x,dy) + \int_{\{|y| \leq r\}} |y|^2 \, \nu(z,dy) \right) \\
		I_2 &:= \left| \int_{\{r<|y|<R\}} (1-e^{iy \xi}+iy \xi \chi(y)) \, \big[\nu(x,dy) - \nu(z,dy)\big] \right| \\
		I_3 &:= 2\nu(x,\{|y| \geq R\}) + 2\nu(z,\{|y| \geq R\}).
	\end{align*}
	The vague continuity in (iii)(b) implies that $I_2 \to 0$ as $z \to x$ for fixed $r,R>0$ with $\nu(x,\{|y|=r\})=0$ and $\nu(x,\{|y|=R\})=0$. Letting first $z \to x$ and then $r \to 0$ and $R \to \infty$, it follows from \ref{appb-1-ii}(c) and \ref{appb-1-ii}(d) that $p(\cdot,\xi)$ is continuous. \par \medskip
	
	\ref{appb-1-i} $\implies$ \ref{appb-1-ii}: By \cite[Thm.\ 4.4]{rs98}, \ref{appb-1-i} implies
	that
	\begin{equation*}
		\forall K \Subset \R^d \, \, :\: \lim_{|\xi| \to 0} \sup_{x \in K} |q(x,\xi)|=0.
	\end{equation*}
	 Moreover, exactly the same reasoning as in \cite[proof of Thm.\ 4.4]{rs98} shows that $\nu(x,\cdot)$, $x \in K$, is tight for any compact set $K$. For $\varphi \in \Cc^{\infty}(\R)$ and $x \in \R$ define 
	 \begin{equation*}
		S_x(\varphi) := \hat{B}(|\cdot-x|^2 \varphi(\cdot-x))(x) = \int_{\R \setminus \{0\}} y^2 \varphi(y) \, \nu(x,dy).
	\end{equation*}
	If we denote by $\mathcal{F}f := \hat{f}$ the Fourier transform of a function $f$, then 
	\begin{equation*}
		\mathcal{F}(|\cdot-x|^2 \varphi(\cdot-x))(\xi) = e^{-ix \xi} \mathcal{F}(|\cdot|^2 \varphi(\cdot))(\xi), \qquad \xi \in \R.
	\end{equation*}
	Since $q$ is locally bounded and $x \mapsto q(x,\xi)$ is continuous for all $\xi \in \R$, an application of the dominated convergence theorem shows that 
	\begin{equation*}
		x \mapsto S_x(\varphi) = - \int_{\R} q(x,\xi) \mathcal{F}(|\cdot-x|^2 \varphi(\cdot-x))(\xi) e^{ix \xi} \, d\xi
	\end{equation*}
	is continuous. Choose $\varphi_k \in \Cc^{\infty}(\R)$ such that $\I_{B(0,1/k)} \leq \varphi_k \leq \I_{B(0,2/k)}$ and $\varphi_{k+1} \leq \varphi_k$. Then $S_x(\varphi_k) \geq S_x(\varphi_{k+1})$ and 
	\begin{equation*}
		S_x(\varphi_k) \leq \int_{|y| \leq 2/k} |y|^2 \, \nu(x,dy) \xrightarrow[]{k \to \infty} 0 \quad \text{for all} \, \, x \in \R.
	\end{equation*}
	Applying Dini's theorem, we find that 
	\begin{equation*}
		\sup_{x \in K} \int_{|y| \leq 1/k} |y|^2 \, \nu(x,dy) \leq \sup_{x \in K} \int y^2 \varphi_k(y) \, \nu(x,dy) = \sup_{x \in K} |S_x(\varphi_k)|\xrightarrow[]{k \to \infty} 0
	\end{equation*}
	for any compact set $K$, and this proves (d).  If we set $\mu(x,dy) := |y-x|^2 \, \nu(x,dy+x)$, then 
	\begin{equation}
		T_x(\varphi) := \hat{B}(|\cdot-x|^2 \varphi(\cdot))(x) =\int_{\R \setminus \{0\}} |y|^2 \varphi(x+y) \, \nu(x,dy) = \int \varphi(y) \, \mu(x,dy) \label{appb-eq6}
	\end{equation}
	for all $\varphi \in \Cc^{\infty}(\R)$. As 
	\begin{align*}
		&\mathcal{F}(|\cdot-x|^2 \varphi(\cdot))(\xi) \\
		&= \quad \frac{1}{2\pi} \left( \int_{\R} y^2 \varphi(y) \,e^{-iy \xi} \, dy -2 x \int_{\R} y \varphi(y) e^{-iy \xi} \, dy + x^2 \int_{\R} \varphi(y) e^{-iy \xi} \, dy \right),
	\end{align*}
	there exists for any compact set $K$ an integrable function $g$ such that $\sup_{x \in K} |\mathcal{F}(|\cdot-x|^2 \varphi)(\xi)| \leq g(\xi)$ for all $\xi \in \R$. Since $x \mapsto q(x,\xi)$ is continuous and locally bounded, the dominated convergence theorem shows that the mapping \begin{equation*}
		x \mapsto T_x(\varphi) = - \int_{\R} q(x,\xi) \mathcal{F}(|\cdot-x|^2 \varphi(\cdot))(\xi) e^{ix \xi} \, d\xi
	\end{equation*} 
	is continuous for all $\varphi \in \Cc^{\infty}(\R)$. By \eqref{appb-eq6} this means that \begin{equation*}
		\int_{\R} \varphi(y) \, \mu(z,dy) \xrightarrow[]{z \to x} \int_{\R} \varphi(y) \, \mu(x,dy) \quad \text{for all} \, \, f \in \Cc^{\infty}(\R), x \in \R.
	\end{equation*}
	Combining this with the fact that the local boundedness of $q$ implies \begin{equation}
		\sup_{z \in K} \int_{|y|\leq R} \, \mu(z,dy)<\infty  \quad \text{for all} \, \, R>0, K \subseteq \R \,\, \text{cpt.} \label{appb-eq7}
	\end{equation}
	cf.\ \eqref{appb-eq5}, we conclude that $\mu(\cdot,dy)$ is vaguely continuous on $(\R,\cB(\R))$.  Using that 
	\begin{align*}
		&\left| \int_{\R} \varphi(y-x) \, \mu(x,dy) - \int_{\R} \varphi(y-z) \, \mu(z,dy) \right|\\
		&\quad \leq \int_{\R} |\varphi(y-x)-\varphi(y-z)| \, \mu(z,dy) 
		+ \left| \int_{\R} \varphi(y-x) \, \big[\mu(z,dy)-  \mu(x,dy)\big] \right|
	\end{align*}
	for all $\varphi \in \Cc(\R)$, it follows from \eqref{appb-eq7} and the vague continuity of $\mu(x,\cdot)$ that \begin{align*}
		\int_{\R} \varphi(y-z) \, \mu(z,dy) \xrightarrow[]{z \to x} \int_{\R} \varphi(y-x) \, \mu(x,dy)
	\end{align*}
	for all $\varphi \in \Cc(\R)$. Since $\nu(x,dy) = \frac{1}{|y|^2} \mu(x,dy-x)$, it is not difficult to see that this implies that $\nu(z,\cdot)$ converges vaguely on $(\R \backslash \{0\},\cB(\R \backslash \{0\}))$ to $\nu(x,\cdot)$ as $z \to x$. To prove continuity of the drift $b$, we note that \begin{align*}
		b(x) = \hat{B}((\cdot-x) \chi(\cdot-x) )(x) 
		&= - \int_{\R} q(x,\xi) e^{ix \xi} \mathcal{F}(((\cdot-x) \chi(\cdot-x))(\xi) \, d\xi \\
		&= - \int_{\R} q(x,\xi) \mathcal{F}(\id(\cdot) \chi(\cdot))(\xi) \, d\xi;
	\end{align*}
	here $\id(y) := y$. Applying the dominated convergence theorem another time, we find that $x \mapsto b(x)$ is continuous. 
	This finishes the proof of the equivalences. Finally, if one (hence, all) conditions of the theorem are satisfied and $\|\hat{B}f\|_{\infty} \leq K \|f\|_{\Cb^{\beta}}$, $f \in C_c^{\infty}(\R^d)$, for some constants $K>0$ and $\beta>0$, then a standard approximation argument yields that $\hat{B}f$ is continuous for every $f \in \Cb^{\beta}(\R^d) \cap \Co(\R^d)$.
\end{proof}

\begin{cor} \label{appb-3} 
	Let $\hat{B}$ be as in \eqref{eq-intdiff2} for a smooth cut-off function $\chi$ with $\I_{B(0,1)} \leq \chi \leq \I_{B(0,2)}$, and denote by $q(x,\xi)$ the symbol of $\hat{B}$, cf.\ \eqref{appb-eq3}. If there is for every compact set $K \subseteq \R^d$ some constant $\alpha \in (0,2)$ such that
	\begin{equation}
		\sup_{x  \in K} |b(x)| + \sup_{x \in K} \left( \int_{\{|y| \leq 1\}} |y|^{\alpha} \, \nu(x,dy)  \right)<\infty,
		\label{eq-intcondition}
	\end{equation}
	and $(\nu(x,\cdot))_{x \in K}$ is tight, then the following statements are equivalent: \begin{enumerate}
		\item $\hat{B}f$ is continuous for every $f \in C_c^{\infty}(\R^d)$,
		\item $x \mapsto q(x,\xi)$ is continuous for all $\xi \in \R^d$,
		\item $x \mapsto b(x)$ is continuous and $x \mapsto \nu(x,\cdot)$ is vaguely continuous.
	\end{enumerate}
\end{cor}

\begin{proof}
	Because of \eqref{eq-intcondition}, the symbol $q$ is locally bounded, cf.\ Remark~\ref{appb-2}, and moreover \ref{appb-1-ii}(c),(d) are clearly satisfied. Thus, the assertion is immediate from Theorem~\ref{appb-1}.
\end{proof}

\end{document}